\documentclass[]{imsart}

\RequirePackage[OT1]{fontenc}
\RequirePackage{amsthm,amsmath}
\RequirePackage[numbers]{natbib}

\usepackage{graphicx}
\usepackage{mathrsfs}
\usepackage{amsfonts}
\usepackage{threeparttable,booktabs}
\usepackage{enumerate}
\usepackage{tikz}
\usepackage{listings}
\usepackage{multirow}
\usepackage{amsfonts}
\usepackage{amssymb}
\usepackage{bbm}
\usepackage{arydshln}

\allowdisplaybreaks

\makeatletter
\def\thm@space@setup{%
  \thm@preskip=0.45cm
  \thm@postskip=\thm@preskip 
}
\makeatother

\startlocaldefs
\numberwithin{equation}{section}
\theoremstyle{plain}
\newtheorem{thm}{Theorem}[section]
\newtheorem{definition}[thm]{Definition}
\newtheorem{assumption}[thm]{Assumption}
\newtheorem{proposition}[thm]{Proposition}
\newtheorem{lemma}[thm]{Lemma}
\theoremstyle{remark}
\newtheorem{rem}{Remark}[section]
\endlocaldefs

\newcommand{\Var}{\mathrm{Var}}

\newcommand{\Cov}{\mathrm{Cov}}

\newcommand\oneStepA{C}
\newcommand\oneStepB{A}
\newcommand\oneStepC{B}

\newcommand\iid{i.i.d.}
\usepackage{enumitem}
\setlist{  
	listparindent=\parindent,
	parsep=0pt,
}

\usepackage{subfigure,sidecap}
\usepackage{wrapfig}
\usepackage{color}

\begin{document}

\begin{frontmatter}
\title{Scalable Monte Carlo Inference and Rescaled Local Asymptotic Normality}
\runtitle{Scalable MC Inference and RLAN}

\begin{aug}
\author{\fnms{Ning} \snm{Ning}\thanksref{e1}\ead[label=e1,mark]{patning@umich.edu}},
\author{\fnms{Edward L.} \snm{Ionides}\ead[label=e2]{ionides@umich.edu}}
\and
\author{\fnms{Ya'acov} \snm{Ritov}\ead[label=e3]{yritov@umich.edu}}

\address{Department of Statistics,
University of Michigan, Ann Arbor.
\printead{e1,e2,e3}}

\runauthor{Ning, Ionides and Ritov}


\end{aug}

\begin{abstract}
 In this paper, we generalize the property of local asymptotic normality (LAN) to an enlarged neighborhood, under the name of rescaled local asymptotic normality (RLAN). We obtain sufficient conditions for a regular parametric model to satisfy RLAN. We show that RLAN supports the construction of a statistically efficient estimator which maximizes a cubic approximation to the log-likelihood on this enlarged neighborhood. In the context of Monte Carlo inference, we find that this maximum cubic likelihood estimator can maintain its statistical efficiency in the presence of asymptotically increasing Monte Carlo error in likelihood evaluation. 
\end{abstract}
  

\end{frontmatter}


\section{Introduction}
\label{sec:Introduction}
\subsection{Background and motivation}
\label{sec:Background_and_motivation}
The classical theory of asymptotics in statistics relies heavily
on certain local approximations to the logarithms of
likelihood ratios, where ``local" is meant to indicate that one looks at parameter values close to a point \citep{le2012asymptotics}. The classic theory of local asymptotic normality (LAN) of \citep{lecam86} concerns a ``local neighborhood" $\{\theta+t_n n^{-1/2}\}$ around a true parameter $\theta$ in an open subset $\Theta \subset \mathbb{R}$, where $t_n$ is a bounded constant and $n$ is the number of observations. 
We suppose the data are modeled as a real-valued sample $(Y_1,\cdots,Y_n)$, for $n\in \mathbb{N}$, from the probability distribution $P_{\theta}$ on the probability space $(\Omega, \mathcal{A}, \mathbb{\mu})$ where $\mu$ is a fixed $\sigma$-finite measure dominating $P_{\theta}$.
Let
\begin{align}
\label{eqn:density_def}
p(\theta)=p(\cdot\;;\theta)=\frac{dP_{\theta}}{d\mu}(\cdot), \quad\quad l(\theta)=\log p(\theta),
\end{align}
be the density and log-likelihood of $P_{\theta}$ respectively. 
Define the log-likelihood of $(Y_1,\cdots,Y_n)$ by
\begin{align}
\label{eqn:Sample_loglikelihood}
\mathbf{l}(\theta)=\sum_{i=1}^n l(Y_i\;;\theta).
\end{align}
The LAN states as follows:
	\begin{equation}
	\label{eq:lan0}
\mathbf{l}(\theta+t_n n^{-1/2})-\mathbf{l}(\theta)
= t_n  S_n(\theta) - \frac{1}{2}t_n^2\mathcal{I}(\theta) + o(1),
	\end{equation}
where $\mathcal{I}(\theta)$ is a finite positive constant, $S_n(\theta)\to N[0,\mathcal{I}(\theta)]$ in distribution, and $o(1)$ is an error term goes to zero in $P_{\theta}$ probability as $n$ goes to infinity.

We form a grid of equally spaced points with separation $n^{-1/2}$ over $\mathbb{R}$ and define $\theta_n^{\ast}$ as the midpoint of the interval into which $\widetilde{\theta}_n$ has fallen, where $\widetilde{\theta}_n$ is a uniformly $\sqrt{n}$-consistent estimator whose existence is established in Theorem $1$ on page $42$ of \cite{bickel1993efficient}.
Then $\theta_n^{\ast}$ is also uniformly $\sqrt{n}$-consistent. 
In practice, for $\theta^\ast_{j,n}=\theta^\ast_n+jn^{-1/2}$ where $j\in\{-1,0,1\}$, the quadratic polynomial of equation \eqref{eq:lan0} can be interpolated by $\big( \theta^\ast_{j,n}, \mathbf{l}(\theta^\ast_{j,n}) \big)_{j}$, and $S_n(\theta_n^{\ast})$ and ${\mathcal{I}}(\theta_n^{\ast})$ can be estimated. Here and in the sequel, we set $\theta^\ast_{0,n}=\theta^\ast_n$.
 The
one-step estimator $\widehat\theta^{\oneStepB}_n$ defined in \cite{bickel1993efficient}, using the estimated $S_n(\theta_n^{\ast})$ and ${\mathcal{I}}(\theta_n^{\ast})$,
\begin{align}
\label{eq:classical_onestep_est}
\widehat\theta^{\oneStepB}_n=&\theta_n^{\ast}+\sqrt{n}\times \frac{S_n(\theta_n^{\ast})}{{\mathcal{I}}(\theta_n^{\ast})},
\end{align}
maximizes the interpolated quadratic approximation to the log-likelihood.

The one-step estimator $\widehat\theta^{\oneStepB}_n$ can be generalized to $\widehat\theta^{\oneStepC}_n$, using the estimated $S_n(\theta_n^{\ast})$ and ${\mathcal{I}}(\theta_n^{\ast})$ generated through a quadratic fit to $\big( \theta^\ast_{j,n}, \mathbf{l}(\theta^\ast_{j,n}) \big)_{j\in \{-J, \cdots, J\}}$ with $J\geq 1$.
When the likelihood can be computed perfectly, there may be little reason to use $\widehat\theta^{\oneStepC}_n$ over $\widehat\theta^{\oneStepB}_n$.
However, when there is Monte Carlo uncertainty or other numerical error in the likelihood evaluation, then $\widehat\theta^{\oneStepC}_n$ with $J>1$ may be preferred.
Taking this idea a step further, we can construct a maximum smoothed likelihood estimator (MSLE), proposed in \citep{ionides2005maximum}, by maximizing a smooth curve fitted to the grid of log-likelihood evaluations $\big( \theta^\ast_{j,n}, \mathbf{l}(\theta^\ast_{j,n}) \big)_{j\in \{-J, \cdots, J\}}$. 
If the smoothing algorithm preserved quadratic functions then the MSLE is asymptotically equivalent to a one-step estimator under the LAN property, while behaving reasonably when the log-likelihood has a substantial deviation from a quadratic.  

The motivations of a rescaled LAN property arise from both the methodological side and the theoretical side: 
\begin{itemize}
\item Monte Carlo likelihood evaluations calculated by simulating from a model, are useful for constructing likelihood-based parameter estimates and confidence intervals, for complex models \citep{diggle84}. 
However, for large datasets any reasonable level of computational effort may result in a non-negligible numerical error in these likelihood evaluations, and the Monte Carlo methods come at the expense of ``poor scalability". 
Specifically, on the classical scale of $n^{-1/2}$, when the number of observations $n$ is large, 
the statistical signal in the likelihood function is asymptotically overcome by Monte Carlo noise in the likelihood evaluations, as the Monte Carlo variance growing with $n$. In line with \cite{ionides17} (page $4$), let the statistical error $\operatorname{SE_{stat}}$ stand for the uncertainty resulting from randomness in the data, viewed as a draw from a statistical model, and let the
Monte Carlo error $\operatorname{SE_{MC}}$ stand for the uncertainty resulting from implementing a Monte Carlo estimator of a statistical model. That is, in the context of Monte Carlo inference, we desire a general methodology to suffice $\operatorname{SE_{MC}^2}/\operatorname{SE_{stat}^2}\to 0$.

\item From a theoretical point of view, LAN has been found to hold in various situations other than regular independent identically distributed ({\iid}) parametric models, including semiparametric models \citep{bickel1996inference}, positive-recurrent Markov chains \citep{hopfner1990local}, stationary hidden Markov models \citep{bickel1996inference}, stochastic block models \citep{bickel2013asymptotic}, and regression models with long memory dependence \citep{hallin1999local}. There are very close linkages between LAN established in regular {\iid} parametric models and LAN established in models such as hidden Markov models and stochastic processes, for example Theorem $1.1$ in \citep{bickel1996inference}.
We anticipate that comparable results could be derived for RLAN.
To motivate future investigations of contexts where RLAN arises, it is necessary to rigorously establish RLAN as a worthwhile statistical property.
\end{itemize}
\begin{figure}[htbp!]
	\includegraphics[width=\textwidth]{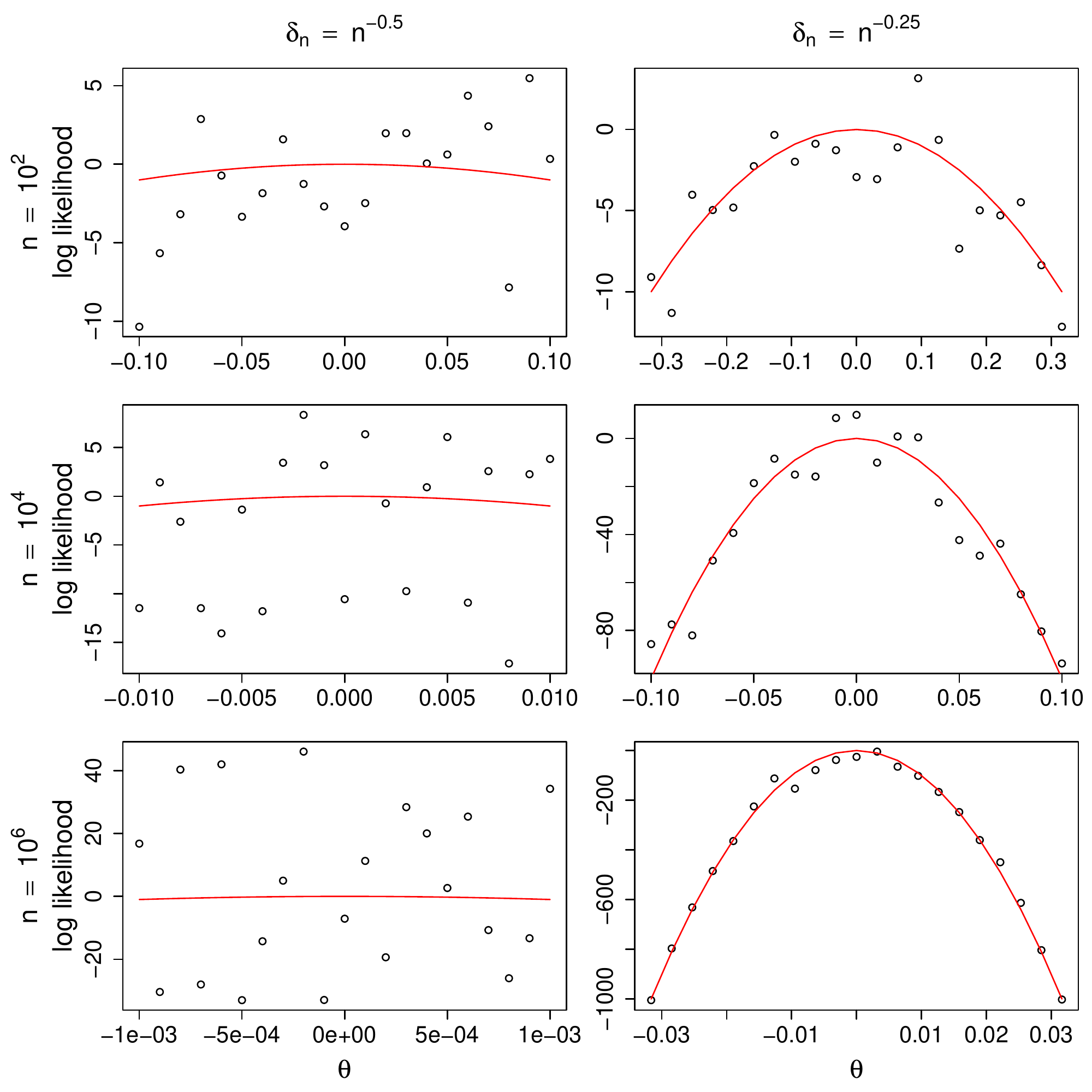}
	\caption{\footnotesize Illustration on the Monte Carlo estimation effects via different values of $n$ and $\delta_n$. The columns correspond to $\delta_n=n^{-0.5},n^{-0.25}$ and the rows correspond to $n=10^2,10^4,10^6$.
		The red solid curve is the log-likelihood given by $\mathbf{l}(\theta)=-n\theta^2$. The black circles are Monte Carlo log-likelihood evaluations, $\overline{\mathbf{l}}(\theta^\ast_{j,n})\sim N[\mathbf{l}(\theta^\ast_{j,n}),n/m]$ with a sample size $m=n^{1/2}$, evaluated at $21$ equally spaced values $\{\theta^\ast_{j,n}\}_j$ in the range $[-\delta_n,\delta_n]$.}
	\label{f:illustration}
\end{figure}

An idea on rescaling is to consider the $n^{-1/4}$ local neighborhoods instead, since the existence of a uniformly $\sqrt{n}$-consistent estimator implies the existence of a uniformly $n^{1/4}$-consistent estimator and then in the $n^{-1/4}$ local neighborhoods we have $\theta_n^{\ast}$ uniformly $n^{1/4}$-consistent. Figure \ref{f:illustration} verifies and illustrates the enlarged neighborhood idea, by showing a Gaussian likelihood function for $n$ observations that is evaluated by a Monte Carlo estimator having variance scaling linearly with $n/m$, where $m=\sqrt{n}$ is the number of Monte Carlo simulations per observation. Monte Carlo evaluations of the log-likelihood were conducted at a grid of points $\{\theta^\ast_{j,n} = \theta_n^\ast+j\delta_n\}_{j\in \{-10,-9,\dots,0,\dots,9,10\}}$, evenly spaced on $[-\delta_n,\delta_n]$, for $\delta_n=n^{-1/2}$ and $\delta_n=n^{-1/4}$ respectively. 
We see from Figure \ref{f:illustration} that on the classical scale of $n^{-1/2}$, when the number of observations $n$ is large, the statistical signal in the likelihood function is asymptotically overcome by Monte Carlo noise in the likelihood evaluations, as the Monte Carlo variance growing with $n$ even though there is just modest growth in the Monte Carlo effort $(m=n^{1/2}).$ However, on the $n^{-1/4}$ scale, the form of the likelihood surface is evident despite the growing Monte Carlo uncertainty. That is, the classical local $n^{-1/2}$ neighborhood does not provide a useful estimate in this limit, but the rescaled local $n^{-1/4}$ neighborhood enables a quadratic likelihood approximation to be successfully fitted. 
Then, we aim to establish the LAN property in the rescaled local $n^{-1/4}$ neighborhood, based on which, an estimator extended the classical one-step estimator in \citep{lecam86} and the MSLE in \citep{ionides2005maximum} can be designed.

\subsection{Our contributions}
\label{sec:Our_contributions}
The contributions of the paper are two-fold:
\begin{enumerate}
\item  We consider a new property:
\begin{definition} \label{def:RLAN}
	Let $\mathbf{P}:=\{P_{\theta}: \theta\in\Theta \}$ be a regular parametric model in the probability space $(\Omega, \mathcal{A}, \mathbb{\mu})$.
	We say that $\mathbf{P}$ has rescaled local asymptotic normality (RLAN) if  uniformly in $\theta \in K$ compact $\subset \Theta$ and $|t_n| \leq M$,
	\begin{equation}
	\label{eq:rlan}
	\mathbf{l}(\theta+t_n n^{-1/4})-\mathbf{l}(\theta)
	=
	n^{1/4} t_n  S_n(\theta) - \frac{1}{2} n^{1/2} t_n^2\mathcal{I}(\theta) + n^{1/4}t_n^3\mathcal{W}(\theta) + \mathcal{O}(1)
	\end{equation}
	and
	\begin{equation}
	\label{eq:lan}
	\mathbf{l}(\theta+t_n n^{-1/2})-\mathbf{l}(\theta)
	= t_n  S_n(\theta) - \frac{1}{2}t_n^2\mathcal{I}(\theta) + o(1),
	\end{equation}
	under $P_{\theta}$, with $\mathcal{I}(\theta)$ being a finite positive constant, $\mathcal{W}(\theta)$ being a finite constant, $S_n\to N[0,\mathcal{I}(\theta)]$ in distribution,  $\mathcal{O}(1)$ denoting an error term bounded in $P_{\theta}$ probability, and $o(1)$ denoting an error term converging to zero in $P_{\theta}$ probability.
\end{definition}
To develop the key ideas, we work in a one-dimensional parameter space.
However, the ideas naturally generalize to $\Theta \subset \mathbb{R}^d$ for $d\ge 1$.
The widely studied property of LAN \citep{lecam86,le2012asymptotics} is defined by \eqref{eq:lan}, so RLAN implies LAN.
The LAN property asserts a quadratic approximation to the log-likelihood function in a neighborhood with scale $n^{-1/2}$, whereas RLAN additionally asserts a cubic approximation on a  $n^{-1/4}$ scale.
	In Section~\ref{sec:rlan:rpm}, we present sufficient conditions for a sequence of {\iid} random variables to satisfy RLAN. 
	Complex dependence structures fall outside the  {\iid} theory of Section~\ref{sec:rlan:rpm}, except in the situation where there is also replication.
	Panel time series analysis via mechanistic models is one situation where  {\iid} replication arises together with model complexity requiring Monte Carlo approaches \citep{breto19,ranjeva19,ranjeva17}.\\

\item  Suppose RLAN (Definition \ref{def:RLAN}) holds. Then, with $\theta^\ast_{j,n} = \theta_n^\ast+jn^{-1/4}$ for $j$ in a finite set $\mathcal{J}$, we can write
\begin{equation}
\label{eqn:MCLEestimator_likeli}
\mathbf{l}(\theta^\ast_{j,n})= \beta_0+\beta_1 (j n^{-1/4}) + \beta_2 (j n^{-1/4})^2 + \beta_3 (j n^{-1/4})^3 + \epsilon_{j,n},
\end{equation}
where $\beta_1=\mathcal{O}(n^{1/2})$, $\beta_2=\mathcal{O}(n)$ and $\beta_3=\mathcal{O}(n)$, and $\epsilon_{j,n}=\mathcal{O}(1)$. In practice, 
the cubic polynomial of equation \eqref{eqn:MCLEestimator_likeli} can be interpolated by $\big( \theta^\ast_{j,n}, \mathbf{l}(\theta^\ast_{j,n}) \big)_{j\in \mathcal{J}}$, and $\{\beta_{\iota}\}_{\iota \in \{1,2,3\}}$ can be estimated.
Based on the linear least squares estimated $\{\beta_{\iota}\}_{\iota \in \{1,2,3\}}$, we can define 
 the maximum cubic log-likelihood estimator (MCLE) when it is finite, as
\begin{equation}
\label{eqn:MCLEestimator_def}
\widehat\theta^{\text{MCLE}}_n = \theta^\ast_n+n^{-1/4}
\underset{\chi\in\mathbb{R}}{\arg \max}\big\{\beta_1 (\chi n^{-1/4})+ \beta_2 (\chi n^{-1/4})^2 + \beta_3 (\chi n^{-1/4})^3 \big\}.
\end{equation}
The MCLE defined above is general, while in this paper we apply it in the context of Monte Carlo inference with {\iid} data samples, under the situation that one does not have access to the likelihood evaluation $\mathbf{l}(\theta^\ast_{j,n})$ but instead can obtain the Monte Carlo likelihood evaluation $\overline{\mathbf{l}}(\theta^\ast_{j,n})$. \\
%
%

We firstly illustrate how $\overline{\mathbf{l}}(\theta^\ast_{j,n})$ may be generated. We suppose the data are modeled as an {\iid} sequence $Y_1,\dots,Y_n$ drawn from a density
\begin{equation}
p(y\;;\theta)=p_Y(y\; ; \theta) = \int p_{Y|X}(y|x \; ; \theta)p_{X}(x\; ; \theta)\, dx.
\end{equation}
For each $Y_i$ and $\theta^\ast_{j,n}$, independent Monte Carlo samples $(X_{i,j}^{(1)},\cdots,X_{i,j}^{(m)})$ for $m\in \mathbb{N}$ are generated from an appropriate probability density function $q(\cdot\;;\theta^\ast_{j,n})$.  
Then, we approximate $p(Y_i\;; \theta^\ast_{j,n})$ with $\overline{p}(Y_i\;;\theta^\ast_{j,n})$
using an importance sampling evaluator,
$$\overline{p}(Y_i\;; \theta^\ast_{j,n})=\frac{1}{m}\sum_{\tau=1}^m p_{Y|X}(Y_i|X_{i,j}^{(\tau)}\;; \theta^\ast_{j,n})\frac{p_X(X_{i,j}^{(\tau)}\;; \theta^\ast_{j,n})}{q(X_{i,j}^{(\tau)}\;; \theta^\ast_{j,n})},$$
which is unbiased by construction.
We construct 
\begin{equation}
\label{eqn:MCestimator}
\overline{\mathbf{l}}(\theta^\ast_{j,n})=\sum_{i=1}^n\ln\overline{p}(Y_i\;; \theta^\ast_{j,n})
\end{equation}
as the estimated log-likelihood.\\

Recalling that $n$ is the number of observations, suppose that $m=m(n)$ is the number of Monte Carlo simulations per observation, and take $\mathcal{O}(\sqrt{n})\ll m(n)\ll \mathcal{O}(n)$.
Then the Monte Carlo log-likelihood theory gives that
\begin{equation}
\label{eqn:MCLEestimator_likeli2}
\overline{\mathbf{l}}(\theta^\ast_{j,n})=\mathbf{l}(\theta^\ast_{j,n})+\gamma(\theta^\ast_{j,n})+\widetilde{\epsilon}_{j,n},
\end{equation}
where $\widetilde{\epsilon}_{j,n}$ is {\iid} such that $\frac{m}{n}\widetilde{\epsilon}_{j,n}$ convergences in distribution to a normal distribution with mean zero and positive finite variance, 
and the bias term $\gamma(\theta^\ast_{j,n})$ satisfies 
\begin{equation}
\label{meta3}
\gamma(\theta^\ast_{j,n})=\gamma(\theta^\ast_{n})+ C_{\gamma}\frac{n}{m} j n^{-1/4}\big(1 +o(1)\big),
\end{equation}
where $C_{\gamma}$ is a finite constant. Plugging equation \eqref{eqn:MCLEestimator_likeli} in equation \eqref{eqn:MCLEestimator_likeli2}, we can obtain that 
\begin{align*}
\overline{\mathbf{l}}(\theta^\ast_{j,n})=&\beta_0+\beta_1 (j n^{-1/4}) + \beta_2 (j n^{-1/4})^2 + \beta_3 (j n^{-1/4})^3 + \epsilon_{j,n}+\gamma(\theta^\ast_{j,n})+\widetilde{\epsilon}_{j,n}.
\end{align*}
Organizing the terms in the above equation, we have the Monte Carlo meta model 
\begin{equation}
\label{meta1}
\begin{split}
\overline{\mathbf{l}}(\theta^\ast_{j,n})= \overline{\beta}_0+\overline{\beta}_1 (j n^{-1/4}) + \overline{\beta}_2 (j n^{-1/4})^2 + \overline{\beta}_3 (j n^{-1/4})^3+\overline{\epsilon}_{j,n},
\end{split}
\end{equation}
where $\theta^\ast_{j,n} = \theta_n^\ast+jn^{-1/4}$ for $j$ in a finite set $\mathcal{J}$, $\overline{\beta}_1=\mathcal{O}(n^{1/2})$, $\overline{\beta}_2=\mathcal{O}(n)$, $\overline{\beta}_3=\mathcal{O}(n)$, and
$\overline{\epsilon}_{j,n}$ is {\iid} such that $\frac{m}{n}\overline{\epsilon}_{j,n}$ convergences in distribution to a normal distribution having mean zero and positive finite variance.
\\

The proposed general methodology takes advantage of asymptotic properties of the likelihood function in an $n^{-1/4}$ neighborhood of the true parameter value. 
In Section~\ref{sec:mc}, we will see that $\widehat\theta^{\text{MCLE}}_n$ is efficient with the desired property that $\operatorname{SE_{MC}^2}/\operatorname{SE_{stat}^2}\to 0$ as the number of data samples $n\to\infty$. The statistically efficient simulation-based likelihood inference is achieved with a computational budget of size essentially $n^{3/2}$ for a dataset of $n$ observations and $\sqrt{n}$ Monte Carlo simulations per observation. In sum, despite substantial Monte Carlo uncertainties involved in the proposed general Monte Carlo based method, MCLE is efficient and able to  scale properly. 
 The proposed methodology sheds light on tackling ``poor scalability" issues in related Monte Carlo based approaches, such as the Monte Carlo adjusted profile methodology of \cite{ionides17} which has been used in various scientific studies \citep{pons-salort18,ranjeva19,ranjeva17,smith17}. However, the extension to profile likelihood estimation, analogous to the LAN-based profile likelihood theory of \citep{murphy00}, is beyond the scope of this paper.
 \end{enumerate}

\subsection{Organization of the paper}
The rest of the paper proceeds as follows:
Section~\ref{sec:rlan:rpm} derives the RLAN property in the context of a regular parametric model, leading to a theorem which is proved in Section~\ref{sec:rlan:rpm:proof};
In Section~\ref{sec:mc}, based on the RLAN property, we investigate the performance of the proposed MCLE in the context of Monte Carlo inference. The notations used throughout this paper are listed in Table~\ref{tab:TableOfNotation}.

\section{RLAN for regular parametric models} \label{sec:rlan:rpm}

In this section, we show that parametric models with sufficient regularity enjoy the RLAN property, for $n$ {\iid} observations.

\subsection{Model setup}
We suppose the data are modeled as a real-valued {\iid} sample $(Y_1,\cdots,Y_n)$, for $n\in \mathbb{N}$, from the probability distribution $P_{\theta}$ on the probability space $(\Omega, \mathcal{A}, \mathbb{\mu})$ where $\mu$ is a fixed $\sigma$-finite measure dominating $P_{\theta}$.
We seek to infer the unknown ``true'' parameter $\theta$ which is situated in an open subset $\Theta \subset \mathbb{R}$. 
We suppose the parameterization $\theta\rightarrow P_{\theta}$ has a density and log-likelihood which can be written as
$$p(\theta)=p(\cdot\;;\theta)=\frac{dP_{\theta}}{d\mu}(\cdot), \quad\quad l(\theta)=\log p(\theta).$$
For $\mathbf{P}=\{P_{\theta}: \theta\in\Theta \}$ being the set of all the probability measures induced by the parameter $\theta$ in the whole parameter set $\Theta$, we metrize $\mathbf{P}$ with the variational distance. Let $v: \mathbf{P}\rightarrow \mathbb{R}$ be a Euclidean parameter, and suppose that $v$ can be identified with the parametric function $q:\Theta\rightarrow \mathbb{R}$ defined by
$$q(\theta)=v(P_{\theta}).$$
Let $\left\lVert \cdot \right\rVert$ stand for the Hilbert norm in $L_2(\mu)$, i.e., $\left\lVert f \right\rVert^2=\int f^2 d\mu$.
It is convenient to view $\mathbf{P}$ as a subset of $L_2(\mu)$ via the embedding 
\begin{equation}
\label{eqn:density_def}
p(\cdot\;;\theta)\rightarrow s(\cdot\;;\theta):=\sqrt{p(\cdot\;;\theta)}.
\end{equation}
where $p(\cdot\;;\theta)$ is the density defined in equation \eqref{eqn:density_def}.

The generalization from the classical theory of LAN in the scale $\mathcal{O}(n^{-1/2})$ to RLAN in the scale $\mathcal{O}(n^{-1/4})$, requires additional smoothness assumptions. We start with the following definition.
\begin{definition}
\label{def:2nd_smooth_point}
We say that $\theta_0$ is a fourth-order regular point of the parametrization $\theta\rightarrow P_{\theta}$, if $\theta_0$ is an interior point of $\Theta$, and 
\begin{enumerate}
\item The map $\theta\rightarrow s(\theta)$ from $\Theta$ to $\mathcal{L}_2(\mu)$ is fourth-order differentiable at $\theta_0$: there exist first-order derivative $\dot{s}(\theta_0)$, second-order derivative $\ddot{s}(\theta_0)$, third-order derivative $\dddot{s}(\theta_0)$, and fourth-order derivative $\ddddot{s}(\theta_0)$
of elements of $\mathcal{L}_2(\mu)$ such that
\begin{align*}
&\left\lVert \frac{s(\theta_0+\delta_n t_n)-s(\theta_0)- \dot{s}(\theta_0) \delta_n t_n }{\delta_n^2} - \frac{1}{2}\ddot{s}(\theta_0) t_n^2 \right\rVert \rightarrow 0,\\
&\left\lVert \frac{s(\theta_0+\delta_n t_n)-s(\theta_0)- \dot{s}(\theta_0) \delta_n t_n -\frac{1}{2}\ddot{s}(\theta_0) \delta_n^2t_n^2}{\delta_n^3} - \frac{1}{3!}\dddot{s}(\theta_0)t_n^3 \right\rVert \rightarrow 0,\\
&\left\lVert \frac{s(\theta_0+\delta_n t_n)-s(\theta_0)- \dot{s}(\theta_0) \delta_n t_n -\frac{1}{2}\ddot{s}(\theta_0) \delta_n^2t_n^2-\frac{1}{3!}\dddot{s}(\theta_0)\delta_n^3 t_n^3}{\delta_n^4} - \frac{1}{4!}\ddddot{s}(\theta_0)t_n^4 \right\rVert \rightarrow 0,
\end{align*}
for any $\delta_n\rightarrow 0$ and $t_n$ bounded.
\item The variable $\dot{\mathbf{I}}(\theta):=2\frac{\dot{s}(\theta)}{s(\theta)}\mathbbm{1}_{\{s(\theta)>0\}}$ has non-zero second moment ${\mathcal{I}}(\theta):=E_{\theta}[\dot{\mathbf{I}}(\theta)]^2$ and non-zero fourth moment.
\item The variable $\ddot{\mathbf{I}}(\theta):=2\frac{\ddot{s}(\theta)}{s(\theta)}\mathbbm{1}_{\{s(\theta)>0\}}$ has non-zero second moment.
\end{enumerate} 
\end{definition}

\begin{assumption}
\label{model_assumption}
We assume the following:
\begin{enumerate} 
\item Every point of $\Theta$ is a fourth-order regular point.
\item The map $\theta \rightarrow \ddddot{s}(\theta)$ is continuous from $\Theta$ to $\mathcal{L}_2(\mu)$.
\item Define $\dddot{\mathbf{I}}(\theta):=2\frac{\dddot{s}(\theta)}{s(\theta)}\mathbbm{1}_{\{s(\theta)>0\}}$ and $\ddddot{\mathbf{I}}(\theta):=2\frac{\ddddot{s}(\theta)}{s(\theta)}\mathbbm{1}_{\{s(\theta)>0\}}$. We have
$$E_{\theta}|\dot{\mathbf{I}}(\theta)|^6<\infty, \quad E_{\theta}|\ddot{\mathbf{I}}(\theta)|^3<\infty,\quad E_{\theta}|\dddot{\mathbf{I}}(\theta)|^2<\infty, \quad E_{\theta}|\ddddot{\mathbf{I}}(\theta)|<\infty.$$
\end{enumerate} 
\end{assumption}

\begin{rem}
We have the following comments regarding Assumption \ref{model_assumption}:
\begin{enumerate}
	\item The conditions $E_{\theta}[\dot{\mathbf{I}}(\theta)]^2\neq 0$, $E_{\theta}[\dot{\mathbf{I}}(\theta)]^4\neq 0$ and $E_{\theta}[\ddot{\mathbf{I}}(\theta)]^2\neq 0$ hold unless for random variables that are zero almost sure. 
	\item The condition ($3$) in Assumption \ref{model_assumption} holds for all bounded random variables, such as truncated normal distributed random variables and finite discrete distributed random variables.
	\item The condition ($3$) in Assumption \ref{model_assumption} holds for some unbounded random variables at least, such as the centered normal distributed random variable with $\theta\in\{1,\sqrt{2},\sqrt{3}\}$, whose probability density function is given by
	$p(y\;;\theta)=\frac{1}{\sqrt{2\pi\theta^2}}e^{-\frac{y^2}{2\theta^2}}.$
	\item In this section, we suppose the data are modeled as a real-valued {\iid} sample $(Y_1,\cdots,Y_n)$. Assumption \ref{model_assumption} still may apply on stochastic process with very desired conditions, such as the basic Ornstein-Uhlenbeck process (\cite{uhlenbeck1930theory}) that is stationary, Gaussian, and Markovian, evolving as  
	$dX_t=-\rho X_tdt+\sigma dW_t$,
	where $\rho$ and $\sigma$ are finite constants, and $W_t$ is the standard Brownian motion with unit variance parameter. Its stationary distribution is the normal distribution with mean $0$ and variance $\theta^2=\frac{\sigma^2}{2\rho}$.
\end{enumerate}
\end{rem}

%

\subsection{The main result}
Define 
\begin{align}
\label{eqn:S_n_theta}
S_n(\theta)=&\frac{1}{\sqrt{n}}\sum_{i=1}^n \dot{\mathbf{I}}(Y_i\;;\theta),\\
\label{eqn:V_n_theta}
V_n(\theta)=&\frac{1}{\sqrt{n}}\sum_{i=1}^n \left[ \ddot{\mathbf{I}}(Y_i \;;\theta) -E_{\theta}\ddot{\mathbf{I}}(\theta) \right],\\
\label{eqn:U_n_theta}
U_n(\theta)=&\frac{1}{\sqrt{n}}\sum_{i=1}^n \left[\dot{\mathbf{I}}^2(Y_i\;;\theta) -{\mathcal{I}}(\theta)\right].
\end{align}
We have the following theorem for RLAN, whose rigorous proof is provided in Section \ref{sec:Proof_of_Main_Theorem}.


\begin{thm}
\label{thm:SLAN}
Suppose that $\mathbf{P}=\{P_{\theta}: \theta\in\Theta \}$ is a regular parametric model satisfying Assumption \ref{model_assumption}. When $\delta_n=\mathcal{O}(n^{-1/4})$, write
\begin{align*}
&\hspace*{-0.3cm} \mathbf{l}(\theta+\delta_n t_n)-\mathbf{l}(\theta)\\
=&t_n \left\{\sqrt{n}\delta_nS_n(\theta) \right\}+t_n^2 \left\{\sqrt{n}\delta_n^2\left[\frac{1}{2}V_n(\theta)-\frac{1}{4}U_n(\theta)\right]  -\frac{1}{2}n\delta_n^2{\mathcal{I}}(\theta)\right\}\\
&+t_n^3 \left\{n\delta_n^3 \left[\frac{1}{12}E_{\theta}[\dot{\mathbf{I}}^3(\theta)]-\frac{1}{8}E_{\theta}[\ddot{\mathbf{I}}(\theta) \dot{\mathbf{I}}(\theta)]+\frac{1}{6}E_{\theta}\left[\dddot{\mathbf{I}}(\theta)\right]\right]\right\}\\
&+t_n^4 \left\{ n \delta_n^4 \left[-\frac{1}{32}E_{\theta}[\dot{\mathbf{I}}^4(\theta)]-\frac{1}{16}E_{\theta}[\ddot{\mathbf{I}}(\theta)]^2-\frac{1}{12} E_{\theta}[\dddot{\mathbf{I}}(\theta)\dot{\mathbf{I}}(\theta)]+\frac{1}{24}E_{\theta}\left[\ddddot{\mathbf{I}}(\theta)\right]\right]\right\}\\
&+R_n(\theta,t_n).
\end{align*} 
Then uniformly in $\theta \in K$ compact $\subset \Theta$ and $|t_n| \leq M$, one has 
$R_n(\theta,t_n)\xrightarrow[]{p} 0$ in $P_{\theta}$ probability, and in the weak topology
$$\mathbf{L}_{\theta}\left(S_n(\theta)\right)\rightarrow N(0,{\mathcal{I}}(\theta)),$$
$$
\mathbf{L}_{\theta}\left(V_n(\theta)\right)\rightarrow N \left(0,\Var_{\theta}[\ddot{\mathbf{I}}(\theta) ]\right),$$
$$
\mathbf{L}_{\theta}\left(U_n(\theta)\right)\rightarrow N \left(0,\Var_{\theta}[\dot{\mathbf{I}}^2(\theta)]\right),$$
where $N(\mu,\sigma^2)$ is the normal distribution with mean $\mu$ and variance $\sigma^2$, and $\mathbf{L}_{\theta}$ is the law under $\theta$.
\end{thm}

\begin{rem}
\label{rem:LAN_RLAN}
The equation in Theorem \ref{thm:SLAN} with $\delta_n=\mathcal{O}(n^{-1/2})$ instead of $\delta_n=\mathcal{O}(n^{-1/4})$ implies the
classical LAN result (Proposition $2$ on page $16$ of \cite{bickel1993efficient}), which can be seen as follows:
\begin{enumerate}
\item For the $t_n$ term, 
$$\left\{\sqrt{n}\delta_nS_n(\theta) \right\}=\frac{2}{\sqrt{n}}\sum_{i=1}^n \frac{\dot{s}(Y_i\;;\theta)}{s(Y_i\;;\theta)}\mathbbm{1}_{\{s(Y_i\;;\theta)>0\}}.$$ 

\item For the $t_n^2$ term, 
\begin{align*}
&\hspace*{-1cm}\left\{\sqrt{n}\delta_n^2\left[\frac{1}{2}V_n(\theta)-\frac{1}{4}U_n(\theta)\right]- \frac{1}{2}n\delta_n^2 {\mathcal{I}}(\theta)\right\}\\
=& \frac{1}{2n}\sum_{i=1}^n \left[ \ddot{\mathbf{I}}(Y_i \;;\theta) -E_{\theta}\ddot{\mathbf{I}}(\theta) \right]-\frac{1}{4n}\sum_{i=1}^n \left[\dot{\mathbf{I}}^2(Y_i\;;\theta) -{\mathcal{I}}(\theta)\right]-\frac{1}{2}{\mathcal{I}}(\theta).
\end{align*}
By Chung's uniform strong law of large number and Lemma \ref{lem:ddot_I2_ui}, one can obtain that
$$\frac{1}{n}\sum_{i=1}^n \left[ \ddot{\mathbf{I}}(Y_i \;;\theta) -E_{\theta}[\ddot{\mathbf{I}}(\theta)]\right]\xrightarrow[]{a.s.} 0,$$
and 
$$\frac{1}{n}\sum_{i=1}^n \left[\dot{\mathbf{I}}^2(Y_i\;;\theta) -{\mathcal{I}}(\theta)\right] \xrightarrow[]{a.s.} 0,$$
uniformly in $\theta \in K$ compact $\subset \Theta$ and $|t_n| \leq M$.
Then the $t_n^2$ coefficient is asymptotically equivalent to 
$$-\frac{1}{2}{\mathcal{I}}(\theta)=-2E_{\theta}\left[ \frac{\dot{s}(Y_i\;;\theta)}{s(Y_i\;;\theta)}\mathbbm{1}_{\{s(Y_i\;;\theta)>0\}}\right]^2.$$

\item For the $t_n^3$ term, since $E_{\theta}[\dot{\mathbf{I}}^3(\theta)]$,  $E_{\theta}[\ddot{\mathbf{I}}(\theta) \dot{\mathbf{I}}(\theta)]$ and $E_{\theta}\left[\dddot{\mathbf{I}}(\theta)\right]$ are finite constants, and $n \delta_n^3 \rightarrow 0$ as $n\rightarrow \infty$,
$$\left\{n\delta_n^3 \left[\frac{1}{12}E_{\theta}[\dot{\mathbf{I}}^3(\theta)]-\frac{1}{8}E_{\theta}[\ddot{\mathbf{I}}(\theta) \dot{\mathbf{I}}(\theta)]+\frac{1}{6}E_{\theta}\left[\dddot{\mathbf{I}}(\theta)\right]\right]\right\}\rightarrow 0.$$

\item Similarly, for the $t_n^4$ term, since
$E_{\theta}[\dot{\mathbf{I}}^4(\theta)]$, $E_{\theta}[\ddot{\mathbf{I}}(\theta)]^2$,  $E_{\theta}\left[\ddddot{\mathbf{I}}(\theta)\right]$ and  $E_{\theta}[\dddot{\mathbf{I}}(\theta)\dot{\mathbf{I}}(\theta)]$ are finite constants, and $n \delta_n^4 \rightarrow 0$ as $n\rightarrow \infty$,
$$\left\{n \delta_n^4 \left[-\frac{1}{32}E_{\theta}[\dot{\mathbf{I}}^4(\theta)]-\frac{1}{16}E_{\theta}[\ddot{\mathbf{I}}(\theta)]^2-\frac{1}{12} E_{\theta}[\dddot{\mathbf{I}}(\theta)\dot{\mathbf{I}}(\theta)]+\frac{1}{24}E_{\theta}\left[\ddddot{\mathbf{I}}(\theta)\right]\right]\right\}\rightarrow 0.$$
\end{enumerate}
\end{rem}

\begin{rem}
Theorem \ref{thm:SLAN} implies the RLAN property in Definition \ref{def:RLAN}. The case $\delta_n=n^{-1/2}$ is already covered in Remark \ref{rem:LAN_RLAN}.
When $\delta_n=n^{-1/4}$, the terms $$t_n^2 \left\{\sqrt{n}\delta_n^2\left[\frac{1}{2}V_n(\theta)-\frac{1}{4}U_n(\theta)\right]\right\}= \mathcal{O}(1)$$ 
and 
$$t_n^4 \left\{n \delta_n^4 \left[-\frac{1}{32}E_{\theta}[\dot{\mathbf{I}}^4(\theta)]-\frac{1}{16}E_{\theta}[\ddot{\mathbf{I}}(\theta)]^2-\frac{1}{12} E_{\theta}[\dddot{\mathbf{I}}(\theta)\dot{\mathbf{I}}(\theta)]+\frac{1}{24}E_{\theta}\left[\ddddot{\mathbf{I}}(\theta)\right]\right]\right\}= \mathcal{O}(1).$$ 
Hence, we have, uniformly in $\theta \in K$ compact $\subset \Theta$ and $|t_n| \leq M$,
\begin{equation} \label{eq:rlan:revisited}
\mathbf{l}(\theta+\delta_n t_n)-\mathbf{l}(\theta)
=
\sqrt{n} \delta_n t_n  S_n - \frac{1}{2} n \delta_n^2 t_n^2\mathcal{I} + n\delta_n^3t_n^3\mathcal{W} + \mathcal{O}(1),
\end{equation}
where $\mathcal{I}={\mathcal{I}}(\theta)$ is a finite positive constant, $S_n\to N[0,\mathcal{I}]$ in distribution, and
$$\mathcal{W}=\left[\frac{1}{12}E_{\theta}[\dot{\mathbf{I}}^3(\theta)]-\frac{1}{8}E_{\theta}[\ddot{\mathbf{I}}(\theta) \dot{\mathbf{I}}(\theta)]+\frac{1}{6}E_{\theta}\left[\dddot{\mathbf{I}}(\theta)\right]\right]$$ is a finite constant.
  
\end{rem}

\section{Developing a proof of Theorem \ref{thm:SLAN}}
\label{sec:rlan:rpm:proof}
In this section, we work toward a proof of Theorem \ref{thm:SLAN}. Throughout this section, we suppose Assumption \ref{model_assumption} holds and consider $\delta_n=\mathcal{O}(n^{-1/4})$. 
We firstly use a truncation method on a Taylor series expansion of  $\mathbf{l}(\theta+\delta_n t_n)-\mathbf{l}(\theta)$ in Section \ref{sec:truncation}, and then conduct preliminary analysis in bounding $\mathbf{l}(\theta+\delta_n t_n)-\mathbf{l}(\theta)$ in Section \ref{sec:Preliminary_Analysis}, both of which prepare for the proof of Theorem \ref{thm:SLAN} in Section \ref{sec:Proof_of_Main_Theorem}. 

\subsection{A truncatated Taylor series remainder}
\label{sec:truncation}
Set 
\begin{equation}
\label{eqn:Tn}
T_n=\left\{\frac{s(\theta+\delta_n t_n)}{s(\theta)}-1-\frac{1}{2}\delta_n t_n \dot{\mathbf{I}}(\theta) \right\}\mathbbm{1}_{\{s(\theta)>0\}}.
\end{equation}
Let $\{T_{ni}\}_{i=1,\cdots, n}$ denote the $n$ {\iid} copies of $T_n$ corresponding to $Y_1,\cdots,Y_n$, and for $\eta \in (0,1)$ define
\begin{equation}
\label{eqn:An_classical}
A_n=\left\{\max_{1\leq i \leq n}\left|T_{ni}+\frac{1}{2}\delta_n t_n \dot{\mathbf{I}}(\theta)\right|<\eta\right\}.
\end{equation}
In the following, we use a truncation method similar to \cite{bickel1993efficient}, but our definition of $T_n$ differs from that in \cite{bickel1993efficient} (page $509$). 
Compared to the corresponding one in \cite{bickel1993efficient}, here $T_n$ additionally incorporates the first-order derivative of $s(\theta)$, since we have to resort to a higher order derivative of $s(\theta)$ for analysis on the $\delta_n=\mathcal{O}(n^{-1/4})$ scale.
By Proposition \ref{prop:A_complement_truncation} following, uniformly in $\theta\in K \subset \Theta$ for $K$ compact and $|t_n|\leq M$, $P_{\theta}(A_n^c)\rightarrow 0$ where $A_n^c$ is the complement of $A_n$. 
On the event $A_n$, we have
\begin{equation*}
	\mathbf{l}(\theta+\delta_n t_n)-\mathbf{l}(\theta)
	=\sum_{i=1}^n \log \left\{\frac{p(Y_i\;;\theta+\delta_n t_n)}{p(Y_i\;;\theta)}\right\}
	=2 \sum_{i=1}^n \log \left( T_{ni}+1+\frac{1}{2}\delta_n t_n \dot{\mathbf{I}}(Y_i\;;\theta)\right).	
\end{equation*}
By a Taylor expansion,
\begin{equation}
\begin{split}
\label{eqn:taylor_usage}
& \mathbf{l}(\theta+\delta_n t_n)-\mathbf{l}(\theta)	\\
&=2\sum_{i=1}^n \left(T_{ni}+\frac{1}{2}\delta_n t_n \dot{\mathbf{I}}(Y_i\;;\theta)\right)-\sum_{i=1}^n \left(T_{ni}+\frac{1}{2}\delta_n t_n \dot{\mathbf{I}}(Y_i\;;\theta)\right)^2\\
&\quad+\frac{2}{3}\sum_{i=1}^n \left(T_{ni}+\frac{1}{2}\delta_n t_n \dot{\mathbf{I}}(Y_i\;;\theta)\right)^3-\frac{1}{2}\sum_{i=1}^n \left(T_{ni}+\frac{1}{2}\delta_n t_n \dot{\mathbf{I}}(Y_i\;;\theta)\right)^4+R_n,
\end{split}
\end{equation}
where $$|R_n|\leq \frac{2C(\eta)}{5}\sum_{i=1}^n \left(T_{ni}+\frac{1}{2}\delta_n t_n \dot{\mathbf{I}}(Y_i\;;\theta)\right)^5,$$
for 
$C(\eta)< (1-\eta)^{-5}$ a finite constant that depends on $\eta$ only.

\subsection{Preliminary analysis}
\label{sec:Preliminary_Analysis}

In this subsection, we develop a sequence of lemmas and propositions needed for the proof in Section~\ref{sec:Proof_of_Main_Theorem}. Specifically, we conduct a series of preliminary analyses to bound the quantities $(T_{ni})^{\alpha}(\delta_n t_n \dot{\mathbf{I}}(Y_i\;;\theta))^{\beta}$ for $\alpha,\beta\in \{0,1,2,3,4,5\}$ such that $\alpha+\beta=5$. 

\begin{lemma}
\label{lem:Tn_second_moment} One has
\begin{align}
&E_{\theta}\left |T_n-\frac{1}{4} (\delta_n t_n)^2\ddot{\mathbf{I}}(\theta)\right|^2=o(\delta_n^4),\label{eqn::Tn_second_moment1}\\
&E_{\theta}\left |T_n-\frac{1}{4} (\delta_n t_n)^2\ddot{\mathbf{I}}(\theta)-\frac{1}{12}(\delta_nt_n)^3\dddot{\mathbf{I}}(\theta)\right|^2=o(\delta_n^6),\label{eqn::Tn_second_moment2}\\
&E_{\theta}\left |T_n-\frac{1}{4} (\delta_n t_n)^2\ddot{\mathbf{I}}(\theta)-\frac{1}{12}(\delta_nt_n)^3\dddot{\mathbf{I}}(\theta)-\frac{1}{48}(\delta_nt_n)^4\ddddot{\mathbf{I}}(\theta)\right|^2=o(\delta_n^8),\label{eqn::Tn_second_moment3}
\end{align}
as $\delta_n\rightarrow 0$, uniformly in $\theta\in K \subset \Theta$ for $K$ compact and $|t_n|\leq M$,
\end{lemma}

\begin{proof}
See Appendix \ref{app:Tn_second_moment}.
\end{proof}

\begin{lemma}
\label{lem:ddot_I2_ui}  One has
\begin{align}
&\lim_{\lambda \rightarrow \infty} \sup_{\theta\in K} E_{\theta}\left[|\ddot{\mathbf{I}}(\theta)|^2 \mathbbm{1}_{\left\{|\ddot{\mathbf{I}}(\theta)|\geq \lambda\right\}}\right]=0,\label{eqn:ddot_I2_ui_1}\\
&\lim_{\lambda \rightarrow \infty} \sup_{\theta\in K} E_{\theta}\left[|\dot{\mathbf{I}}(\theta)|^4 \mathbbm{1}_{\left\{|\dot{\mathbf{I}}(\theta)|\geq \lambda\right\}}\right]=0,\label{eqn:ddot_I2_ui_2}\\
&\lim_{\lambda \rightarrow \infty} \sup_{\theta\in K} E_{\theta}\left[\left |\ddddot{\mathbf{I}}(\theta)\right| \mathbbm{1}_{\left\{\left |\ddddot{\mathbf{I}}(\theta)\right|\geq \lambda\right\}}\right]=0,\label{eqn:ddot_I2_ui_3}\\
&\lim_{\lambda \rightarrow \infty} \sup_{\theta\in K} E_{\theta}\left[|\dddot{\mathbf{I}}(\theta)\dot{\mathbf{I}}(\theta)| \mathbbm{1}_{\left\{|\dddot{\mathbf{I}}(\theta)\dot{\mathbf{I}}(\theta)|\geq \lambda\right\}}\right]=0.\label{eqn:ddot_I2_ui_4}
\end{align}
\end{lemma}

\begin{proof}
See Appendix \ref{app:ddot_I2_ui}.
\end{proof}

Recall that $\{T_{ni}\}_{i=1,\cdots, n}$ denote the $n$ {\iid} copies of $T_n$ corresponding to $Y_1,\cdots,Y_n$. Define
\begin{equation}
\label{eqn:An}
\widetilde{A}_n(\epsilon)=\big\{\max_{1\leq i \leq n}|T_{ni}|<\epsilon\big\},
\end{equation}
for every $\epsilon>0$.

\begin{proposition}
\label{prop:A_complement}
Uniformly in $\theta\in K \subset \Theta$ for $K$ compact and $|t_n|\leq M$,
$$P_{\theta}(\widetilde{A}_n^c)\rightarrow 0,$$
where $\widetilde{A}_n^c$ is the complement of $\widetilde{A}_n$. 
\end{proposition}

\begin{proof}
We firstly note that 
$$P_{\theta}(\widetilde{A}_n^c)\leq \sum_{i=1}^n P_{\theta}(|T_{ni}|\geq \epsilon)=nP_{\theta}(|T_{n}|\geq \epsilon).$$
Then, it suffices to show that
$$P_{\theta}(|T_{n}|\geq \epsilon)=o(1/n).$$
But
\begin{equation}
\begin{split}
\label{eqn:A_complement1}
P_{\theta}(|T_{n}|\geq \epsilon)
\leq & P_{\theta}\left(\left |T_n- \frac{1}{4}(\delta_n t_n)^2 \ddot{\mathbf{I}}(\theta)\right|\geq \frac{1}{2}\epsilon \right)
+P_{\theta}\left(\left |\frac{1}{4}(\delta_n t_n)^2 \ddot{\mathbf{I}}(\theta)\right|\geq \frac{1}{2}\epsilon \right)\\
\leq & \frac{4}{\epsilon^2}E_{\theta}\left |T_n- \frac{1}{4}(\delta_n t_n)^2 \ddot{\mathbf{I}}(\theta)\right|^2
+\frac{4}{\epsilon^2}E_{\theta}\left |\frac{1}{4}(\delta_n t_n)^2 \ddot{\mathbf{I}}(\theta)\right|^2 \mathbbm{1}_{\left\{\left |(\delta_n t_n)^2 \ddot{\mathbf{I}}(\theta)\right|\geq \epsilon\right\}}\\
\leq & o(\delta_n^4)+\frac{1}{4\epsilon^2}(\delta_n t_n)^4E_{\theta}|\ddot{\mathbf{I}}(\theta)|^2 \mathbbm{1}_{\left\{\left |(\delta_n t_n)^2 \ddot{\mathbf{I}}(\theta)\right|\geq \epsilon\right\}}\\
=&o(\delta_n^4),
\end{split}
\end{equation}
where the second to the last step is by Lemma \ref{lem:Tn_second_moment}, and the last step is by Lemma \ref{lem:ddot_I2_ui}.
\end{proof}

Recall that for $\eta \in (0,1)$
\begin{equation}
A_n=\left\{\max_{1\leq i \leq n}\left|T_{ni}+\frac{1}{2}\delta_n t_n \dot{\mathbf{I}}(\theta)\right|<\eta\right\}.
\end{equation}
\begin{proposition}
	\label{prop:A_complement_truncation}
	Uniformly in $\theta\in K \subset \Theta$ for $K$ compact and $|t_n|\leq M$,
	$$P_{\theta}(A_n^c)\rightarrow 0,$$
	where $A_n^c$ is the complement of $A_n$. 
\end{proposition}

\begin{proof}
	We firstly note that 
	$$P_{\theta}(A_n^c)\leq \sum_{i=1}^n P_{\theta}\left(\left|T_{ni}+\frac{1}{2}\delta_n t_n \dot{\mathbf{I}}(\theta)\right|\geq \eta\right)=nP_{\theta}\left(\left|T_{n}+\frac{1}{2}\delta_n t_n \dot{\mathbf{I}}(\theta)\right|\geq \eta\right).$$
	Then, it suffices to show that
	$$P_{\theta}\left(\left|T_{n}+\frac{1}{2}\delta_n t_n \dot{\mathbf{I}}(\theta)\right|\geq \eta\right)=o(1/n).$$
	But
	\begin{equation*}
	\begin{split}
	P_{\theta}\left(\left|T_{n}+\frac{1}{2}\delta_n t_n \dot{\mathbf{I}}(\theta)\right|\geq \eta\right)
	\leq & P_{\theta}\left(\left |T_n\right|\geq \frac{1}{2}\eta \right)
	+P_{\theta}\left(\left |\frac{1}{2}\delta_n t_n \dot{\mathbf{I}}(\theta)\right|\geq \frac{1}{2}\eta \right)\\
	\leq & o(\delta_n^4)+\frac{1}{\eta^2}(\delta_n t_n)^4E_{\theta}|\dot{\mathbf{I}}(\theta)|^4 \mathbbm{1}_{\left\{\left |\delta_n t_n \dot{\mathbf{I}}(\theta)\right|\geq \epsilon\right\}}\\
	=&o(\delta_n^4),
	\end{split}
	\end{equation*}
	where the second to the last step is by equation \eqref{eqn:A_complement1} and the last step is by Lemma \ref{lem:ddot_I2_ui}.
\end{proof}

\begin{proposition} 
\label{prop:sum_T_secondmoment}
For any $r\geq 0$, we have
$$\sum_{i=1}^n \left|T_{ni}^2-\frac{1}{16}\delta_n^4t_n^4 E_{\theta}[\ddot{\mathbf{I}}(\theta)]^2\right|^{1+r}\xrightarrow[]{p} 0,$$
uniformly in $|t_n| \leq M$ and in $\theta\in K \subset \Theta$ for $K$ compact,
in $P_{\theta}$ probability.
\end{proposition}

\begin{proof}
See Appendix \ref{app:sum_T_secondmoment}
\end{proof}

\begin{proposition} 
\label{prop:T_3to5}
We have, for any $k\geq 1$,
	$$\sum_{i=1}^n  |T_{ni}|^{2+k} \xrightarrow[]{p} 0,$$
uniformly in $|t_n| \leq M$ and in $\theta\in K \subset \Theta$ for $K$ compact,
	in $P_{\theta}$ probability.
\end{proposition}

\begin{proof}
Let $a(\epsilon')$ be a real valued function on any $\epsilon'>0$ satisfying $\left(\frac{\epsilon'}{a(\epsilon')+\frac{1}{16}n\delta_n^4 t_n^4 E_{\theta}[\ddot{\mathbf{I}}(\theta)]^2}\right)\in (0,1)$. The proof can be completed by noting that by Propositions \ref{prop:A_complement} and \ref{prop:sum_T_secondmoment}, uniformly in $|t_n| \leq M$ and in $\theta\in K \subset \Theta$ for $K$ compact,
$$P_{\theta}\left(\sum_{i=1}^n T_{ni}^2> a(\epsilon')+\frac{1}{16}n\delta_n^4 t_n^4 E_{\theta}[\ddot{\mathbf{I}}(\theta)]^2\right)\rightarrow 0,$$
$$P_{\theta}\left(\max_{1\leq i \leq n} |T_{ni}|>\left(\frac{\epsilon'}{a(\epsilon')+\frac{1}{16}n\delta_n^4 t_n^4 E_{\theta}[\ddot{\mathbf{I}}(\theta)]^2}\right)^{1/k}\right)\rightarrow 0.$$
\end{proof}

\begin{proposition} \label{prop:dotS_fifth}
	We have, for $m\in \{5,6\}$,
	$$\sum_{i=1}^n \left|\left(\delta_n t_n \dot{\mathbf{I}}(Y_i\;;\theta)\right)^m \right| \xrightarrow[]{p} 0,$$
uniformly in $|t_n| \leq M$ and in $\theta\in K \subset \Theta$ for $K$ compact,
	in $P_{\theta}$ probability.
\end{proposition}

\begin{proof}
By Assumption \ref{model_assumption} and Markov inequality, for $|t_n| \leq M$, $\epsilon'>0$, and $m\in \{5,6\}$,
	\begin{equation*}
	\begin{split}
P_{\theta}\left( \sum_{i=1}^n \left|\delta_n t_n \dot{\mathbf{I}}(Y_i\;;\theta)\right|^m>\epsilon' \right) \leq & \frac{1}{\epsilon'} E_{\theta}\left( \sum_{i=1}^n \left|\delta_n t_n \dot{\mathbf{I}}(Y_i\;;\theta)\right|^m \right)\\
 \leq & \frac{n (\delta_n)^m M^m}{\epsilon'} E_{\theta}\left|\dot{\mathbf{I}}(Y_i\;;\theta)\right|^m\\
 \rightarrow & \;  0.
	\end{split}
	\end{equation*}
\end{proof}

\begin{proposition}
\label{prop:Tmulti_dotS}
We have, for any $l\geq 2$ and any $k\geq 1$,
	$$\sum_{i=1}^n \left| T_{ni}^l \left(\delta_n t_n \dot{\mathbf{I}}(Y_i\;;\theta)\right)^k \right| \xrightarrow[]{p} 0,$$
uniformly in $|t_n| \leq M$ and in $\theta\in K \subset \Theta$ for $K$ compact,
	in $P_{\theta}$ probability.
\end{proposition}

\begin{proof}
	See Appendix \ref{app:Tmulti_dotS}.
\end{proof}

\begin{proposition}
\label{prop:4th_LLN}
	We have, uniformly in $|t_n| \leq M$ and in $\theta\in K \subset \Theta$ for $K$ compact,
	\begin{align*}
	&\sum_{i=1}^n \left(\delta_n t_n \dot{\mathbf{I}}(Y_i\;;\theta)\right)^4-n\delta_n^4 t_n^4 E_{\theta}[\dot{\mathbf{I}}(\theta)]^4 \xrightarrow[]{a.s.} 0,\\
	&\sum_{i=1}^n \left(\delta_n t_n\right)^4 \ddddot{\mathbf{I}}(Y_i\;;\theta)-n\delta_n^4 t_n^4 E_{\theta}\left[\ddddot{\mathbf{I}}(\theta) \right] \xrightarrow[]{a.s.} 0,\\
&\sum_{i=1}^n \left(\delta_n t_n\right)^4 \dddot{\mathbf{I}}(Y_i\;;\theta)\dot{\mathbf{I}}(Y_i\;;\theta)-n\delta_n^4 t_n^4 E_{\theta}[\dddot{\mathbf{I}}(\theta)\dot{\mathbf{I}}(\theta)] \xrightarrow[]{a.s.} 0,\\
&\sum_{i=1}^n \left(\delta_n t_n\right)^4 \ddot{\mathbf{I}}(Y_i\;;\theta)\dot{\mathbf{I}}^2(Y_i\;;\theta)-n\delta_n^4 t_n^4 E_{\theta}[\ddot{\mathbf{I}}(\theta)\dot{\mathbf{I}}^2(\theta)] \xrightarrow[]{a.s.} 0.
	\end{align*}
\end{proposition}

\begin{proof}
We complete the proof by  noting that, by
Chung's uniform strong law of large number which can be seen in Theorem A.$7.3$ in \cite{bickel1993efficient} and Lemma \ref{lem:ddot_I2_ui}, one has
\begin{align*}
&\left(\frac{1}{n}\sum_{i=1}^n \left(\dot{\mathbf{I}}(Y_i\;;\theta)\right)^4
-  E_{\theta}[\dot{\mathbf{I}}(\theta)]^4\right) \xrightarrow[]{a.s.} 0,\\
&\left(\frac{1}{n}\sum_{i=1}^n \ddddot{\mathbf{I}}(Y_i\;;\theta)
-  E_{\theta}\left[\ddddot{\mathbf{I}}(\theta) \right]\right) \xrightarrow[]{a.s.} 0,\\
&\left(\frac{1}{n}\sum_{i=1}^n \dddot{\mathbf{I}}(Y_i\;;\theta)\dot{\mathbf{I}}(Y_i\;;\theta)
-  E_{\theta}[\dddot{\mathbf{I}}(\theta)\dot{\mathbf{I}}(\theta)]\right) \xrightarrow[]{a.s.} 0.
\end{align*}
Furthermore, noting that $\ddot{\mathbf{I}}(Y_i\;;\theta)\dot{\mathbf{I}}^2(Y_i\;;\theta)\leq \frac{1}{2}\ddot{\mathbf{I}}^2(Y_i\;;\theta)+\frac{1}{2}\dot{\mathbf{I}}^4(Y_i\;;\theta)$,
 $E_{\theta}\ddot{\mathbf{I}}^2(Y_i\;;\theta)<\infty$ and $E_{\theta}\dot{\mathbf{I}}^4(Y_i\;;\theta)<\infty$ (Assumption \ref{model_assumption}), by 
 Chung's uniform strong law of large number and Lemma \ref{lem:ddot_I2_ui}, one has
 $$\left(\frac{1}{n}\sum_{i=1}^n \ddot{\mathbf{I}}(Y_i\;;\theta)\dot{\mathbf{I}}^2(Y_i\;;\theta)
 -  E_{\theta}[\ddot{\mathbf{I}}(\theta)\dot{\mathbf{I}}^2(\theta)]\right) \xrightarrow[]{a.s.} 0.$$
\end{proof}

\begin{proposition}
\label{prop:summation_Tn}
We have $$\sum_{i=1}^n \bigg(T_{ni} - \frac{1}{4}\delta_n^2 t_n^2 \ddot{\mathbf{I}}(Y_i \;;\theta)-\frac{1}{12}(\delta_n t_n)^3\dddot{\mathbf{I}}(Y_i \;;\theta)-\frac{1}{48}(\delta_n t_n)^4\ddddot{\mathbf{I}}(Y_i \;;\theta)\bigg) \xrightarrow[]{p} 0,$$
uniformly in $|t_n| \leq M$ and in $\theta\in K \subset \Theta$ for $K$ compact,
in $P_{\theta}$ probability.
\end{proposition}

\begin{proof}
For any $\epsilon'>0$, by Markov inequality, Jensen's inequality, and Lemma \ref{lem:Tn_second_moment},
\begin{align*}
& P_{\theta}\left\{\left|\sum_{i=1}^n \bigg(T_{ni} - \frac{1}{4}\delta_n^2 t_n^2 \ddot{\mathbf{I}}(Y_i \;;\theta)-\frac{1}{12}(\delta_n t_n)^3\dddot{\mathbf{I}}(Y_i \;;\theta)-\frac{1}{48}(\delta_n t_n)^4\ddddot{\mathbf{I}}(Y_i \;;\theta)\bigg)\right|>\epsilon'\right\}\\
&\leq \frac{n}{\epsilon'}E_{\theta}\left|T_n-\frac{1}{4} (\delta_n t_n)^2 \ddot{\mathbf{I}}(\theta)-\frac{1}{12}(\delta_n t_n)^3\dddot{\mathbf{I}}(\theta)-\frac{1}{48}(\delta_n t_n)^4\ddddot{\mathbf{I}}(\theta)\right|
\\
& \rightarrow \;  0.
\end{align*}

\end{proof}

\begin{proposition}
\label{prop:Tn1_DotS2}
We have
	$$\sum_{i=1}^n \left| \left(T_{ni}- \frac{1}{4}\delta_n^2 t_n^2 \ddot{\mathbf{I}}(Y_i \;;\theta)\right) \left(\delta_n t_n \dot{\mathbf{I}}(Y_i\;;\theta)\right)^2 \right| \xrightarrow[]{p} 0,$$
uniformly in $|t_n| \leq M$ and in $\theta\in K \subset \Theta$ for $K$ compact,
	in $P_{\theta}$ probability.
\end{proposition}

\begin{proof}
For any $\epsilon'>0$, by Markov inequality, H\"older's inequality, and Lemma \ref{lem:Tn_second_moment},
\begin{align*}
&\hspace*{-1.5cm} P_{\theta}\left\{\sum_{i=1}^n \left| \left(T_{ni}- \frac{1}{4}\delta_n^2 t_n^2 \ddot{\mathbf{I}}(Y_i \;;\theta)\right) \left(\delta_n t_n \dot{\mathbf{I}}(Y_i\;;\theta)\right)^2 \right|>\epsilon'\right\}\\
\leq &\frac{n}{\epsilon'}E_{\theta}\left|\left(T_{n}- \frac{1}{4}\delta_n^2 t_n^2 \ddot{\mathbf{I}}(\theta)\right) \left(\delta_n t_n \dot{\mathbf{I}}(\theta)\right)^2\right|\\
\leq & \frac{n}{\epsilon'}\left[E_{\theta}\left(T_{n}- \frac{1}{4}\delta_n^2 t_n^2 \ddot{\mathbf{I}}(\theta)\right)^2\right]^{1/2}\left[E_{\theta} \left(\delta_n t_n \dot{\mathbf{I}}(\theta)\right)^4\right]^{1/2}
\\
\rightarrow & \;  0.
\end{align*}
\end{proof}

\begin{proposition}
\label{prop:Tn1_DotS1}
We have
	$$\sum_{i=1}^n \left| \left(T_{ni}- \frac{1}{4}\delta_n^2 t_n^2 \ddot{\mathbf{I}}(Y_i \;;\theta)-\frac{1}{12}(\delta_n t_n)^3\dddot{\mathbf{I}}(Y_i \;;\theta)\right) \left(\delta_n t_n \dot{\mathbf{I}}(Y_i\;;\theta)\right) \right| \xrightarrow[]{p} 0,$$
uniformly in $|t_n| \leq M$ and in $\theta\in K \subset \Theta$ for $K$ compact,
	in $P_{\theta}$ probability.
\end{proposition}

\begin{proof}
	For any $\epsilon'>0$, by Markov inequality, H\"older's inequality, and Lemma \ref{lem:Tn_second_moment},
	\begin{align*}
	&\hspace*{-1.5cm} P_{\theta}\left\{\sum_{i=1}^n \left| \left(T_{ni}- \frac{1}{4}\delta_n^2 t_n^2 \ddot{\mathbf{I}}(Y_i \;;\theta)-\frac{1}{12}(\delta_n t_n)^3\dddot{\mathbf{I}}(Y_i \;;\theta)\right) \left(\delta_n t_n \dot{\mathbf{I}}(Y_i\;;\theta)\right) \right|>\epsilon'\right\}\\
	\leq &\frac{n}{\epsilon'}E_{\theta}\left|\left(T_{n}- \frac{1}{4}\delta_n^2 t_n^2 \ddot{\mathbf{I}}(\theta)-\frac{1}{12}(\delta_n t_n)^3\dddot{\mathbf{I}}(\theta)\right) \left(\delta_n t_n \dot{\mathbf{I}}(\theta)\right)\right|\\
	\leq & \frac{n}{\epsilon'}\left[E_{\theta}\left(T_{n}- \frac{1}{4}\delta_n^2 t_n^2 \ddot{\mathbf{I}}(\theta)-\frac{1}{12}(\delta_n t_n)^3\dddot{\mathbf{I}}(\theta)\right)^2\right]^{1/2}\left[E_{\theta} \left(\delta_n t_n \dot{\mathbf{I}}(\theta)\right)^2\right]^{1/2}
	\\
	\rightarrow & \;  0.
	\end{align*}
\end{proof}

\begin{proposition}
\label{prop:Tn1_DotS34}
We have, for any $k\geq 3$,
	$$\sum_{i=1}^n \left| T_{ni} \left(\delta_n t_n \dot{\mathbf{I}}(Y_i\;;\theta)\right)^k \right| \xrightarrow[]{p} 0,$$
uniformly in $|t_n| \leq M$ and in $\theta\in K \subset \Theta$ for $K$ compact,
	in $P_{\theta}$ probability.
\end{proposition}

\begin{proof}
Note that
\begin{equation*}
\begin{split}
\sum_{i=1}^n \left| T_{ni} \left(\delta_n t_n \dot{\mathbf{I}}(Y_i\;;\theta)\right)^k \right|
 \leq & \frac{1}{2}\sum_{i=1}^n (T_{ni})^2 \left|\delta_n t_n \dot{\mathbf{I}}(Y_i\;;\theta)\right|^{2k-5}  +\frac{1}{2}\sum_{i=1}^n \left|\delta_n t_n \dot{\mathbf{I}}(Y_i\;;\theta)\right|^{5}.
\end{split}
\end{equation*}
By Proposition \ref{prop:Tmulti_dotS} and Proposition \ref{prop:dotS_fifth}, we complete the proof.
\end{proof}

\begin{proposition}
\label{prop:U_CLT}
	Define 
$$U(\theta)=U(Y_i\;;\theta)=\dot{\mathbf{I}}^2(Y_i\;;\theta)-{\mathcal{I}}(\theta).$$ 
We have
\begin{equation}
\label{eqn:prop_CLT}
\mathbf{L}_{\theta}\left(\frac{1}{\sqrt{n}}\sum_{i=1}^n U(Y_i\;;\theta)\right)\rightarrow N(0,\Var_{\theta}[\dot{\mathbf{I}}^2(\theta)]),
\end{equation}
uniformly in $\theta \in K$ for compact $K \in \Theta$ in the weak topology, where $N(0,\Var_{\theta}[\dot{\mathbf{I}}^2(\theta)])$ is the
normal distribution with mean $0$ and covariance matrix $\Var_{\theta}[\dot{\mathbf{I}}^2(\theta)]$. 
\end{proposition}

\begin{proof}
See Appendix \ref{app:U_CLT}.
\end{proof}

\begin{proposition}
\label{prop:V_CLT}
	Define 
$$V(Y_i\;;\theta)=\ddot{\mathbf{I}}(Y_i \;;\theta) -E_{\theta}[\ddot{\mathbf{I}}(Y_i \;;\theta)].$$ 
We have
\begin{equation}
\label{eqn:prop_CLT2}
\mathbf{L}_{\theta}\left(\frac{1}{\sqrt{n}}\sum_{i=1}^n V(Y_i\;;\theta)\right)\rightarrow N(0,\Var[\ddot{\mathbf{I}}(\theta)]),
\end{equation}
uniformly in $\theta \in K$ for compact $K \in \Theta$ in the weak topology, where $N(0,\Var[\ddot{\mathbf{I}}(\theta)]))$ is the
normal distribution with mean $0$ and covariance matrix $\Var[\ddot{\mathbf{I}}(\theta)]$. 
\end{proposition}

\begin{proof}
See Appendix \ref{app:V_CLT}.
\end{proof}

\begin{proposition}
\label{prop:DotS_3}
We have, uniformly in $|t_n| \leq M$ and in $\theta\in K \subset \Theta$ for $K$ compact,
in $P_{\theta}$ probability,
\begin{align}
&\sum_{i=1}^n \delta_n^3 t_n^3 \dot{\mathbf{I}}^3(Y_i\;;\theta)-\sum_{i=1}^n \delta_n^3 t_n^3 E_{\theta}[\dot{\mathbf{I}}^3(Y_i\;;\theta)]\xrightarrow[]{p} 0,\label{eqn:DotS_3_1}\\
&\sum_{i=1}^n \delta_n^3 t_n^3 \dddot{\mathbf{I}}(Y_i \;;\theta) -\sum_{i=1}^n \delta_n^3 t_n^3 E_{\theta}[\dddot{\mathbf{I}}(Y_i \;;\theta) ]\xrightarrow[]{p} 0,\label{eqn:DotS_3_3}\\
&\sum_{i=1}^n \delta_n^3 t_n^3 \ddot{\mathbf{I}}(Y_i \;;\theta) \dot{\mathbf{I}}(Y_i\;;\theta)-\sum_{i=1}^n \delta_n^3 t_n^3 E_{\theta}[\ddot{\mathbf{I}}(Y_i \;;\theta) \dot{\mathbf{I}}(Y_i\;;\theta)]\xrightarrow[]{p} 0,\label{eqn:DotS_3_2}.
\end{align}
\end{proposition}

\begin{proof}
We firstly prove equation \eqref{eqn:DotS_3_1}. 
Since $E_{\theta}(\dot{\mathbf{I}}^6(\theta))$ is finite by Assumption \ref{model_assumption}, we have $\Var_{\theta} (\dot{\mathbf{I}}^3(\theta))<\infty.$
By Chebyshev's inequality and independence of data samples, we have that for any $\epsilon'>0$
\begin{align*}
&  P_{\theta}\left[\left|\sum_{i=1}^n \left(\delta_n t_n \dot{\mathbf{I}}(Y_i\;;\theta)\right)^3-\sum_{i=1}^n E_{\theta}\left(\delta_n t_n \dot{\mathbf{I}}(Y_i\;;\theta)\right)^3\right|>\epsilon'\right]\\
&\leq  \frac{1}{(\epsilon')^2} \Var_{\theta}\left[\sum_{i=1}^n \left(\delta_n t_n \dot{\mathbf{I}}(Y_i\;;\theta)\right)^3\right]
= \frac{n \delta_n^6 M^6}{(\epsilon')^2}\Var_{\theta} \left(\dot{\mathbf{I}}^3(\theta)\right)
\rightarrow  \;  0,
\end{align*}
uniformly in $\theta \in K$ and $|t_n| \leq M$. 
Equation \eqref{eqn:DotS_3_3} can be proved similarly, since $E_{\theta} \left(\dddot{\mathbf{I}}(\theta)\right)^2$ is finite by Assumption \ref{model_assumption}.

Next we prove equation \eqref{eqn:DotS_3_2}.
By H\"older's inequality, $$E_{\theta}(\ddot{\mathbf{I}}(\theta) \dot{\mathbf{I}}(\theta))^2\leq  \left(E_{\theta}[\ddot{\mathbf{I}}(\theta)]^{2\times 3/2} \right)^{2/3}\left(E_{\theta}[\dot{\mathbf{I}}(\theta)]^{2\times 3} \right)^{1/3}.$$
By Assumption \ref{model_assumption} we know that $\ddot{\mathbf{I}}(\theta)$ has finite third moment and $\dot{\mathbf{I}}(\theta)$ has finite sixth moment, which implies $$\Var_{\theta} \left(\ddot{\mathbf{I}}(Y_i \;;\theta) \dot{\mathbf{I}}(Y_i\;;\theta)\right)<\infty.$$
By Chebyshev's inequality and independence of data samples, we have that $\epsilon'>0$
\begin{align*}
&P_{\theta}\left[\left|\sum_{i=1}^n \delta_n^3 t_n^3 \ddot{\mathbf{I}}(Y_i \;;\theta) \dot{\mathbf{I}}(Y_i\;;\theta)-\sum_{i=1}^n \delta_n^3 t_n^3 E_{\theta}[\ddot{\mathbf{I}}(Y_i \;;\theta) \dot{\mathbf{I}}(Y_i\;;\theta)]\right|>\epsilon'\right]\\
&\leq  \frac{1}{(\epsilon')^2} \Var_{\theta}\left[\sum_{i=1}^n \delta_n^3 t_n^3 \ddot{\mathbf{I}}(Y_i \;;\theta) \dot{\mathbf{I}}(Y_i\;;\theta)\right]
= \frac{n \delta_n^6  M^6}{(\epsilon')^2}\Var_{\theta} \left(\ddot{\mathbf{I}}(Y_i \;;\theta) \dot{\mathbf{I}}(Y_i\;;\theta)\right)
\rightarrow  \;  0,
\end{align*}
uniformly in $\theta \in K$ and $|t_n| \leq M$.
\end{proof}

\subsection{Proof of Theorem \ref{thm:SLAN}}
\label{sec:Proof_of_Main_Theorem}
By Propositions \ref{prop:T_3to5}, \ref{prop:dotS_fifth}, \ref{prop:Tmulti_dotS}, 
and \ref{prop:Tn1_DotS34},  
equation \eqref{eqn:taylor_usage} can be rewritten as
\begin{equation}
\begin{split}
\label{eqn:scale_likelihood2}
&\hspace*{-0.3cm}\mathbf{l}(\theta+\delta_n t_n)-\mathbf{l}(\theta)\\
=&2\sum_{i=1}^n T_{ni}+\sum_{i=1}^n\delta_n t_n \dot{\mathbf{I}}(Y_i\;;\theta)-\sum_{i=1}^n T_{ni}^2-\sum_{i=1}^n \left(T_{ni}\delta_n t_n \dot{\mathbf{I}}(Y_i\;;\theta)\right)\\
&-\frac{1}{4}\sum_{i=1}^n \left(\delta_n t_n \dot{\mathbf{I}}(Y_i\;;\theta)\right)^2
+\frac{1}{12}\sum_{i=1}^n \left(\delta_n t_n \dot{\mathbf{I}}(Y_i\;;\theta)\right)^3+\frac{1}{2}\sum_{i=1}^n T_{ni} \left(\delta_n t_n \dot{\mathbf{I}}(Y_i\;;\theta)\right)^2\\
&-\frac{1}{32}\sum_{i=1}^n \left(\delta_n t_n \dot{\mathbf{I}}(Y_i\;;\theta)\right)^4+R_n(\theta,t_n).
\end{split}
\end{equation}
Here and in the sequel, $R_n(\theta,t_n)\xrightarrow[]{p} 0$ in $P_{\theta}$ probability,
uniformly in $\theta \in K$ compact $\subset \Theta$ and $|t_n| \leq M$, while the explicit expression of $R_n(\theta,t_n)$ may change line by line.
The proof proceeds by tackling the terms in equation \eqref{eqn:scale_likelihood2} one by one, which all hold uniformly in $\theta \in K$ compact $\subset \Theta$ and $|t_n| \leq M$, as follows:
\begin{enumerate}
\item For the term $2\sum_{i=1}^n T_{ni}$, by the result of Proposition \ref{prop:summation_Tn}, , $$\sum_{i=1}^n \bigg(T_{ni} - \frac{1}{4}\delta_n^2 t_n^2 \ddot{\mathbf{I}}(Y_i \;;\theta)-\frac{1}{12}(\delta_n t_n)^3\dddot{\mathbf{I}}(Y_i \;;\theta)-\frac{1}{48}(\delta_n t_n)^4\ddddot{\mathbf{I}}(Y_i \;;\theta)\bigg) \xrightarrow[]{p} 0.$$
Then by Propositions \ref{prop:DotS_3} and \ref{prop:4th_LLN}, 
$$\sum_{i=1}^n \bigg(T_{ni} - \frac{1}{4}\delta_n^2 t_n^2 \ddot{\mathbf{I}}(Y_i \;;\theta)-\frac{1}{12}(\delta_n t_n)^3E_{\theta}\left[\dddot{\mathbf{I}}(\theta)\right]-\frac{1}{48}(\delta_n t_n)^4E_{\theta}\left[\ddddot{\mathbf{I}}(\theta)\right]\bigg) \xrightarrow[]{p} 0.$$
By the result of Proposition \ref{prop:V_CLT}, we have, 
\begin{align*}
2\sum_{i=1}^n T_{ni}
=&\frac{1}{2}t_n^2 \sqrt{n}\delta_n^2 V_n(\theta)+\frac{1}{2}\delta_n^2t_n^2 n E_{\theta}[\ddot{\mathbf{I}}(\theta)]+\frac{1}{6}(\delta_n t_n)^3nE_{\theta}\left[\dddot{\mathbf{I}}(\theta)\right]\\
&+\frac{1}{24}(\delta_n t_n)^4nE_{\theta}\left[\ddddot{\mathbf{I}}(\theta)\right]+R_n(\theta,t_n),
\end{align*}
where $V_n(\theta)$ is defined in equation \eqref{eqn:V_n_theta} and  distributed as 
$$
\mathbf{L}_{\theta}\left(V_n(\theta)\right)\rightarrow N \left(0,\Var_{\theta}(\ddot{\mathbf{I}}(\theta) )\right).$$

\item For the term $\sum_{i=1}^n\delta_n t_n \dot{\mathbf{I}}(Y_i\;;\theta)$, 
 we have
\begin{align*}
	\sum_{i=1}^n\delta_n t_n \dot{\mathbf{I}}(Y_i\;;\theta)=& t_n \sqrt{n}\delta_n S_n(\theta),
\end{align*}
where $S_n(\theta)$ is defined in equation \eqref{eqn:S_n_theta} and distributed as
$$\mathbf{L}_{\theta}\left(S_n(\theta)\right)\rightarrow N(0,{\mathcal{I}}(\theta)),$$
in the weak topology, by Proposition $2.2$ of \cite{bickel1993efficient}.\\

\item For the term $-\sum_{i=1}^n T_{ni}^2$, by Proposition \ref{prop:sum_T_secondmoment}, we have
\begin{align*}
-\sum_{i=1}^n T_{ni}^2=& -\sum_{i=1}^n \left(T_{ni}^2-\frac{1}{16}\delta_n^4t_n^4 E_{\theta}[\ddot{\mathbf{I}}(\theta)]^2\right)-\sum_{i=1}^n \frac{1}{16}\delta_n^4t_n^4 E_{\theta}[\ddot{\mathbf{I}}(\theta)]^2\\
=&-\frac{1}{16}n \delta_n^4 t_n^4 E_{\theta}[\ddot{\mathbf{I}}(\theta)]^2+R_n(\theta,t_n).
\end{align*}

\item For the term $-\sum_{i=1}^n \left(T_{ni}\delta_n t_n \dot{\mathbf{I}}(Y_i\;;\theta)\right)$, we have
\begin{align*}
&-\sum_{i=1}^n \left(T_{ni}\delta_n t_n \dot{\mathbf{I}}(Y_i\;;\theta)\right)\\&=	-\sum_{i=1}^n \left(T_{ni}- \frac{1}{4}\delta_n^2 t_n^2 \ddot{\mathbf{I}}(Y_i \;;\theta)- \frac{1}{12}\delta_n^3 t_n^3 \dddot{\mathbf{I}}(Y_i \;;\theta)\right) \left(\delta_n t_n \dot{\mathbf{I}}(Y_i\;;\theta)\right)\\
&\quad-\frac{1}{4}\sum_{i=1}^n \delta_n^3 t_n^3 \ddot{\mathbf{I}}(Y_i \;;\theta) \dot{\mathbf{I}}(Y_i\;;\theta)- \frac{1}{12}\sum_{i=1}^n\delta_n^4 t_n^4 \dddot{\mathbf{I}}(Y_i \;;\theta)\dot{\mathbf{I}}(Y_i\;;\theta).
\end{align*}
By Propositions \ref{prop:4th_LLN}, \ref{prop:Tn1_DotS1} and \ref{prop:DotS_3}, we have
\begin{align*}
&-\sum_{i=1}^n \left(T_{ni}\delta_n t_n \dot{\mathbf{I}}(Y_i\;;\theta)\right)\\
&=-\frac{1}{4}n\delta_n^3 t_n^3 E_{\theta}[\ddot{\mathbf{I}}(\theta) \dot{\mathbf{I}}(\theta)]-\frac{1}{12}n\delta_n^4 t_n^4 E_{\theta}[\dddot{\mathbf{I}}(\theta)\dot{\mathbf{I}}(\theta)]+R_n(\theta,t_n).
\end{align*}

\item For the term $-\frac{1}{4}\sum_{i=1}^n \left(\delta_n t_n \dot{\mathbf{I}}(Y_i\;;\theta)\right)^2$, 
 we have
\begin{align*}
-\frac{1}{4}\sum_{i=1}^n \left(\delta_n t_n \dot{\mathbf{I}}(Y_i\;;\theta)\right)^2=&-\frac{1}{4}\delta_n^2 t_n^2\sum_{i=1}^n \left[\dot{\mathbf{I}}^2(Y_i\;;\theta) -{\mathcal{I}}(\theta)\right]-\frac{1}{4}\delta_n^2 t_n^2 n {\mathcal{I}}(\theta)\\
=&-\frac{1}{4}t_n^2 \delta_n^2 \sqrt{n} U_n(\theta)-\frac{1}{4}\delta_n^2 t_n^2 n {\mathcal{I}}(\theta),
\end{align*}
where 
$U_n(\theta)$ is defined in equation \eqref{eqn:U_n_theta} and is distributed (Proposition \ref{prop:U_CLT}) as 
$$
\mathbf{L}_{\theta}\left(U_n(\theta)\right)\rightarrow N \left(0,\Var_{\theta}[\dot{\mathbf{I}}^2(\theta)]\right).$$

\item For the term $\frac{1}{12}\sum_{i=1}^n \left(\delta_n t_n \dot{\mathbf{I}}(Y_i\;;\theta)\right)^3$, by Proposition \ref{prop:DotS_3}, we have
\begin{align*}
\frac{1}{12}\sum_{i=1}^n \left(\delta_n t_n \dot{\mathbf{I}}(Y_i\;;\theta)\right)^3=&
\frac{1}{12}\sum_{i=1}^n \delta_n^3 t_n^3\left(\dot{\mathbf{I}}^3(Y_i\;;\theta)-E_{\theta}[\dot{\mathbf{I}}^3(Y_i\;;\theta)]\right)\\
&+\frac{1}{12}n \delta_n^3 t_n^3E_{\theta}[\dot{\mathbf{I}}^3(\theta)]\\
=&\frac{1}{12}n \delta_n^3 t_n^3E_{\theta}[\dot{\mathbf{I}}^3(\theta)]+R_n(\theta,t_n).
\end{align*}

\item For the term $\frac{1}{2}\sum_{i=1}^n T_{ni} \left(\delta_n t_n \dot{\mathbf{I}}(Y_i\;;\theta)\right)^2$, by Proposition \ref{prop:Tn1_DotS2}, we have
\begin{align*}
\frac{1}{2}\sum_{i=1}^n T_{ni} \left(\delta_n t_n \dot{\mathbf{I}}(Y_i\;;\theta)\right)^2=&\frac{1}{2}\sum_{i=1}^n \left(T_{ni}- \frac{1}{4}\delta_n^2 t_n^2 \ddot{\mathbf{I}}(Y_i \;;\theta)\right) \left(\delta_n t_n \dot{\mathbf{I}}(Y_i\;;\theta)\right)^2 \\
&+\frac{1}{8}\sum_{i=1}^n\delta_n^3 t_n^3 \ddot{\mathbf{I}}(Y_i \;;\theta)\dot{\mathbf{I}}(Y_i\;;\theta)\\
=&\frac{1}{8}n\delta_n^3 t_n^3 E_{\theta}[\ddot{\mathbf{I}}(\theta)\dot{\mathbf{I}}(\theta)]+R_n(\theta,t_n).
\end{align*}

\item For the term $-\frac{1}{32}\sum_{i=1}^n \left(\delta_n t_n \dot{\mathbf{I}}(Y_i\;;\theta)\right)^4$, by Proposition \ref{prop:4th_LLN}, we have
\begin{align*}
& -\frac{1}{32}\sum_{i=1}^n \left(\delta_n t_n \dot{\mathbf{I}}(Y_i\;;\theta)\right)^4\\
&=-\frac{1}{32}\sum_{i=1}^n \left(\delta_n t_n \dot{\mathbf{I}}(Y_i\;;\theta)\right)^4+\frac{1}{32}t_n^4  n \delta_n^4 E_{\theta}[\dot{\mathbf{I}}^4(\theta)]-\frac{1}{32}t_n^4  n \delta_n^4 E_{\theta}[\dot{\mathbf{I}}^4(\theta)]\\
&=-\frac{1}{32}t_n^4 n \delta_n^4 E_{\theta}[\dot{\mathbf{I}}^4(\theta)]+R_n(\theta,t_n).
\end{align*}
\end{enumerate}

Now, we can rewrite equation \eqref{eqn:scale_likelihood2} as
\begin{align*}
&\hspace*{-0.3cm} \mathbf{l}(\theta+\delta_n t_n)-\mathbf{l}(\theta)\\
=&\frac{1}{2}t_n^2 \sqrt{n}\delta_n^2 V_n(\theta)+\frac{1}{2}\delta_n^2t_n^2 n E_{\theta}[\ddot{\mathbf{I}}(\theta)]+\frac{1}{6}(\delta_n t_n)^3nE_{\theta}\left[\dddot{\mathbf{I}}(\theta)\right]+\frac{1}{24}(\delta_n t_n)^4nE_{\theta}\left[\ddddot{\mathbf{I}}(\theta)\right]\\
&+ t_n \sqrt{n}\delta_n S_n(\theta)-\frac{1}{16}n \delta_n^4 t_n^4 E_{\theta}[\ddot{\mathbf{I}}(\theta)]^2-\frac{1}{4}n\delta_n^3 t_n^3 E_{\theta}[\ddot{\mathbf{I}}(\theta) \dot{\mathbf{I}}(\theta)]-\frac{1}{12}n\delta_n^4 t_n^4 E_{\theta}[\dddot{\mathbf{I}}(\theta)\dot{\mathbf{I}}(\theta)]
\\
&-\frac{1}{4}t_n^2 \delta_n^2 \sqrt{n} U_n(\theta)-\frac{1}{4}\delta_n^2 t_n^2 n {\mathcal{I}}(\theta)+\frac{1}{12}n \delta_n^3 t_n^3E_{\theta}[\dot{\mathbf{I}}^3(\theta)]+\frac{1}{8}n\delta_n^3 t_n^3 E_{\theta}[\ddot{\mathbf{I}}(\theta)\dot{\mathbf{I}}(\theta)]
\\
&-\frac{1}{32}t_n^4 n \delta_n^4 E_{\theta}[\dot{\mathbf{I}}^4(\theta)]+R_n(\theta,t_n).
\end{align*}
Reorganizing the terms, we have
\begin{align*}
	&\hspace*{-0.3cm} \mathbf{l}(\theta+\delta_n t_n)-\mathbf{l}(\theta)\\
	=&t_n \left\{\sqrt{n}\delta_nS_n(\theta) \right\}+t_n^2 \left\{\sqrt{n}\delta_n^2\left[\frac{1}{2}V_n(\theta)-\frac{1}{4}U_n(\theta)\right]+ n\delta_n^2 \left[\frac{1}{2}E_{\theta}[\ddot{\mathbf{I}}(\theta)]-\frac{1}{4}{\mathcal{I}}(\theta)\right]\right\}\\
	&+t_n^3 \left\{n\delta_n^3 \left[\frac{1}{12}E_{\theta}[\dot{\mathbf{I}}^3(\theta)]-\frac{1}{8}E_{\theta}[\ddot{\mathbf{I}}(\theta) \dot{\mathbf{I}}(\theta)]+\frac{1}{6}E_{\theta}\left[\dddot{\mathbf{I}}(\theta)\right]\right]\right\}\\
	&+t_n^4 \left\{ n \delta_n^4 \left[-\frac{1}{32}E_{\theta}[\dot{\mathbf{I}}^4(\theta)]-\frac{1}{16}E_{\theta}[\ddot{\mathbf{I}}(\theta)]^2-\frac{1}{12} E_{\theta}[\dddot{\mathbf{I}}(\theta)\dot{\mathbf{I}}(\theta)]+\frac{1}{24}E_{\theta}\left[\ddddot{\mathbf{I}}(\theta)\right]\right]\right\}\\
	&+R_n(\theta,t_n).
\end{align*}
 
We complete the proof by noting that differentiating $\int s^2(\theta) d\mu=1$ with respect to $\theta$ yields $\int \dot{s}(\theta)s(\theta) d\mu=0$, and further differentiating with respect to $\theta$ yields 
\begin{equation*}
\int \ddot{s}(\theta)s(\theta) d\mu+\int \dot{s}^2(\theta) d\mu=0,
\end{equation*}
which gives
\begin{equation*}
2E_{\theta}[\ddot{\mathbf{I}}(\theta)]=4\int \ddot{s}(\theta) s(\theta) d\mu=-4\int \dot{s}^2(\theta) d\mu=-{\mathcal{I}}(\theta).
\end{equation*}

\section{Properties of the MCLE methodology}
\label{sec:mc}

In this section, we firstly elaborate how $\{\overline{\beta}_{\iota}\}_{\iota \in \{1,2,3\}}$ in equation \eqref{meta1} may be obtained in practice. 
Recalling that $\theta^\ast_{j,n}-\theta^{\ast}_n=jn^{-1/4}$ for $j$ in a finite set $\mathcal{J}$. We take the classical setting that $\mathcal{J}=\{-J,-J+1,\cdots,0,\cdots,J-1,J\}$ for $J$ being an integer greater than $1$, where we exclude $J=1$ since we need at least $4$ values to interpolate a cubic polynomial curve. Write 
\begin{align*}
\overline{\mathbf{Y}}=\left( \begin{array}{c} \overline{\mathbf{l}}(\theta^\ast_{-J,n}) \\ \vdots\\ \overline{\mathbf{l}}(\theta^\ast_{n}) \\ \vdots\\\overline{\mathbf{l}}(\theta^\ast_{J,n}) \end{array} \right), \quad \quad
\overline{\mathbf{X}}=\left( \begin{array}{cccc} 1 & (-Jn^{-1/4}) &(-Jn^{-1/4})^2&(-Jn^{-1/4})^3 \\
\vdots & \vdots & \vdots& \vdots\\
1 & 0 & 0 & 0\\
\vdots & \vdots & \vdots & \vdots\\
1 & (Jn^{-1/4}) &(Jn^{-1/4})^2 &(Jn^{-1/4})^3 \end{array} \right),\end{align*}
$\overline{\beta}=(\overline{\beta}_0,\overline{\beta}_1,\overline{\beta}_2,\overline{\beta}_3)^T$, and $\overline{\epsilon}=(\overline{\epsilon}_{-J,n}, \cdots, \overline{\epsilon}_0 , \cdots,\overline{\epsilon}_{J,n})^T$, where the superscript ``$T$" stands for the transpose operation.
Now we fit ``data'' $\overline{\mathbf{Y}}$ to $\overline{\mathbf{X}}$ by linear regression 
$$\overline{\mathbf{Y}}=\overline{\mathbf{X}}\overline{\beta}+\overline{\epsilon}.$$
By least-squares estimation, we obtain the estimation of regression coefficients as
\begin{equation}
\label{eqn:beta_LSE}
\widehat{\beta}=(\widehat{\beta}_0,\widehat{\beta}_1,\widehat{\beta}_2,\widehat{\beta}_3)^T=\left(\overline{\mathbf{X}}^T\overline{\mathbf{X}}\right)^{-1}\overline{\mathbf{X}}^T\overline{\mathbf{Y}},
\end{equation}
where
\begin{align*}
&\overline{\mathbf{X}}^T\overline{\mathbf{X}}\\
&=\left( \begin{array}{cccc} J & \sum_{j=-J}^J (jn^{-1/4}) &\sum_{j=-J}^J (jn^{-1/4})^2 &\sum_{j=-J}^J (jn^{-1/4})^3 \\
\sum_{j=-J}^J (jn^{-1/4}) & \sum_{j=-J}^J (jn^{-1/4})^2 &\sum_{j=-J}^J (jn^{-1/4})^3 &\sum_{j=-J}^J (jn^{-1/4})^4\\
\sum_{j=-J}^J (jn^{-1/4})^2 & \sum_{j=-J}^J (jn^{-1/4})^3 &\sum_{j=-J}^J (jn^{-1/4})^4 &\sum_{j=-J}^J (jn^{-1/4})^5\\
\sum_{j=-J}^J (jn^{-1/4})^3 & \sum_{j=-J}^J (jn^{-1/4})^4 &\sum_{j=-J}^J (jn^{-1/4})^5 &\sum_{j=-J}^J (jn^{-1/4})^6 \end{array} \right).
\end{align*}

Before we investigate the order of $\widehat{\beta}$, let us firstly explore
 the orders of the determinant and adjugate matrix of $\overline{\mathbf{X}}^T\overline{\mathbf{X}}$ in the following two lemmas.
\begin{lemma}
\label{lemma:order_determinate}
For $J$ being a fixed integer greater than $1$, the determinant of $\overline{\mathbf{X}}^T\overline{\mathbf{X}}$, denoted as $\operatorname{det}\left(\overline{\mathbf{X}}^T\overline{\mathbf{X}}\right)$, is of order $\mathcal{O}(n^{-3})$.
\end{lemma}

\begin{proof} 
See Appendix \ref{app:order_determinate}.
\end{proof}

\begin{lemma}
\label{lemma:order_adjugate}
For $J$ being a fixed integer greater than $1$, the adjugate of $\overline{\mathbf{X}}^T\overline{\mathbf{X}}$, denoted as $\operatorname{adj}\left(\overline{\mathbf{X}}^T\overline{\mathbf{X}}\right)$, has that
 $$\operatorname{adj}\left(\overline{\mathbf{X}}^T\overline{\mathbf{X}}\right)_{22}= \mathcal{O}(n^{-5/2})\quad\text{and}\quad \operatorname{adj}\left(\overline{\mathbf{X}}^T\overline{\mathbf{X}}\right)_{33}= \mathcal{O}(n^{-2}).$$
\end{lemma}

\begin{proof}
See Appendix \ref{app:order_adjugate}.
\end{proof}

Recall that in equation \eqref{meta1}, we have $\overline{\beta}_2=\mathcal{O}(n)$ and $\overline{\beta}_3=\mathcal{O}(n)$, and then we can see that 
the contribution from the term $\overline{\beta}_3 (j n^{-1/4})^3$ in finding the desired maximizer in the interval $[-Jn^{-1/4},Jn^{-1/4}]$ is asymptotically negligible on the $\mathcal{O}(n^{-1/2})$ scale.
The MCLE is therefore close to the maximizer of the quadratic approximation. Given that the coefficient of the quadratic term is negative by the RLAN property in equation \eqref{eq:rlan}, we have
\begin{equation}
\label{eqn:MCLEestimator_formula}
\widehat\theta^{\text{MCLE}}_n=\theta_n^{\ast}+\frac{\widehat{\beta}_1}{-2\widehat{\beta}_2} + o(n^{-1/2}),
\end{equation}
where $\widehat{\beta}_1$ and $\widehat{\beta}_2$ are given in equation \eqref{eqn:beta_LSE}.
In the following theorem, we compare the performance of $\widehat\theta^{\text{MCLE}}_n$ with the generalized estimator $\widehat{\theta}^{\oneStepA}_n$, which is defined on the $\delta_n=n^{-1/4}$ scale as follows:
\begin{align}
\label{eqn:rescale_onestep}
\widehat{\theta}^{\oneStepA}_n=&\theta_n^{\ast}+\delta_n \times \frac{\sqrt{n}\delta_nS_n(\theta_n^{\ast})}{n\delta_n^2 \mathcal{I}(\theta_n^{\ast})}.
\end{align}
Note that, $\widehat{\theta}^{\oneStepA}_n$ is the generalization of $\widehat\theta^{\oneStepC}_n$ (Section \ref{sec:Background_and_motivation}), for the reason that equation \eqref{eqn:rescale_onestep} with $\delta_n=n^{-1/2}$ instead of $\delta_n=n^{-1/4}$,
gives $\widehat\theta^{\oneStepC}_n$.


\begin{thm}
	\label{thm:MCestimator}
	Suppose that the data are modeled as an {\iid} sequence $Y_1,\dots,Y_n$ drawn from a regular parametric model $\mathbf{P}=\{P_{\theta}: \theta\in\Theta \}$ satisfying Assumption \ref{model_assumption}. 
Take $m=m(n)$ Monte Carlo simulations per observation where $\mathcal{O}(\sqrt{n})\ll m\ll \mathcal{O}(n)$, and take $\mathcal{J}=\{-J,-J+1,\cdots,0,\cdots,J-1,J\}$ for $J$ being a fixed integer greater than $1$.
Then the maximum cubic log-likelihood estimator $\widehat\theta^{\text{MCLE}}_n$ is efficient.
\end{thm}

\begin{proof}
By the classical results of linear regression (see, e.g., equation $(2.13)$ on page $12$ of \cite{rao2008linear}), we know that $\mathbb{E}(\widehat{\beta}_1)=\overline{\beta}_1$ and $\mathbb{E}(\widehat{\beta}_2)=\overline{\beta}_2$. 
By the RLAN property (Definition \ref{def:RLAN}) we can see that the coefficients $\overline{\beta}_1$ and $\overline{\beta}_2$ are of order $\sqrt{n}$ and $n$ respectively (see Section \ref{sec:Our_contributions} ($2$) for illustration). 
Recall that by least-squares estimation we have 
$$
\widehat{\beta}=(\widehat{\beta}_0,\widehat{\beta}_1,\widehat{\beta}_2,\widehat{\beta}_3)^T=\left(\overline{\mathbf{X}}^T\overline{\mathbf{X}}\right)^{-1}\overline{\mathbf{X}}^T\overline{\mathbf{Y}}.
$$
By Lemma \ref{lemma:order_determinate} we have $\operatorname{det}\left(\overline{\mathbf{X}}^T\overline{\mathbf{X}}\right)=\mathcal{O}(n^{-3})$, and by Lemma \ref{lemma:order_adjugate} we can see that  $\operatorname{adj}\left(\overline{\mathbf{X}}^T\overline{\mathbf{X}}\right)_{22}=\mathcal{O}(n^{-5/2})$ and $\operatorname{adj}\left(\overline{\mathbf{X}}^T\overline{\mathbf{X}}\right)_{33}=\mathcal{O}(n^{-2})$.
Hence, by the formula that $$\left(\overline{\mathbf{X}}^T\overline{\mathbf{X}}\right)^{-1}_{ij}=\left(\frac{1}{\operatorname{det}\left(\overline{\mathbf{X}}^T\overline{\mathbf{X}}\right)}\operatorname{adj}\left(\overline{\mathbf{X}}^T\overline{\mathbf{X}}\right)\right)_{ij},$$
we have
\begin{equation}
\label{eqn:order_XX} \left(\overline{\mathbf{X}}^T\overline{\mathbf{X}}\right)^{-1}_{22}=\mathcal{O}(n^{1/2})\quad \text{and} \quad \left(\overline{\mathbf{X}}^T\overline{\mathbf{X}}\right)^{-1}_{33}=\mathcal{O}(n).
\end{equation}
By the classical results of linear regression (see, e.g., equation $(2.15)$ on page $12$ of \cite{rao2008linear}), we have that
$$\Var(\widehat{\beta}_1)=\big(\overline{\mathbf{X}}^T\overline{\mathbf{X}}\big)_{22}^{-1}\Var(\overline{\epsilon}_{j,n})\quad\text{and}\quad \Var(\widehat{\beta}_2)=\big(\overline{\mathbf{X}}^T\overline{\mathbf{X}}\big)_{33}^{-1}\Var(\overline{\epsilon}_{j,n}).$$
Then by the delta method (see, Proposition $9.32$ in \cite{keener2011theoretical})
 and by equation \eqref{eqn:order_XX}, together with the fact that $\frac{m}{n}\overline{\epsilon}_{j,n}$ convergences in distribution to a normal distribution having mean zero and positive finite variance, one obtains that 
$$\Var(\widehat{\beta}_1)= \mathcal{O}\left(n^{1/2} \times \frac{n}{m}\right)\quad\text{and}\quad\Var(\widehat{\beta}_2)= \mathcal{O}\left(n \times \frac{n}{m}\right).$$  
 
By the covariance inequality that for two random variables $\zeta_1$ and $\zeta_2$ their covariance
$$\Cov(\zeta_1, \zeta_2)\leq \sqrt{\Var(\zeta_1)\Var(\zeta_2)},$$
we have 
\begin{align*}
\frac{\overline{\beta}_1^2}{4\overline{\beta}_2^2}\bigg[ \frac{\Var(\widehat{\beta}_1)}{\overline{\beta}_1^2}-2\frac{\Cov(\widehat{\beta}_1,\widehat{\beta}_2)}{\overline{\beta}_1\overline{\beta}_2}+\frac{\Var(\widehat{\beta}_2)}{\overline{\beta}_2^2}\bigg]= \mathcal{O}\left( \frac{n}{n^2} \bigg[\frac{n^{1/2}\times \frac{n}{m}}{n}\bigg] \right)= \mathcal{O}\left( \frac{n^{1/2}\times\frac{1}{m}}{n} \right).
\end{align*}
By the delta method, one has that as $n\rightarrow \infty$,
$
\frac{n}{n^{1/2}\times\frac{1}{m}}\left(\frac{\widehat{\beta}_1}{-2\widehat{\beta}_2}\right)
$
convergences in distribution to a normal distribution having positive finite variance.
Recall that under the conditions imposed in Theorem \ref{thm:MCestimator}, we have equation \eqref{eqn:MCLEestimator_formula}.
Then, in the context of Monte Carlo inference under investigation, we have $$\operatorname{SE_{MC}^2}=\mathcal{O}\left( \Var\left(\frac{\widehat{\beta}_1}{-2\widehat{\beta}_2}\right)\right) = \mathcal{O}\left( \frac{n^{1/2}\times\frac{1}{m}}{n} \right).$$
In the context of Monte Carlo inference under investigation, the statistical standard error $\operatorname{SE_{stat}}$ is the standard deviation of the MCLE constructed with no Monte Carlo error.
Thus $\operatorname{SE_{stat}^2}$ is given by the variance of the classical efficient estimator $\widehat{\theta}^{\oneStepA}_n$, i.e.,
\begin{equation}
\label{eqn:se_thetaA}
\operatorname{SE_{stat}^2}=\Var\left[\widehat{\theta}^{\oneStepA}_n \right]=\frac{1}{n{\mathcal{I}}(\theta)}= \mathcal{O}\left( \frac{1}{n}\right).
\end{equation}
Hence, given $\mathcal{O}(\sqrt{n})\ll m(n)\ll \mathcal{O}(n)$, we have $\operatorname{SE_{MC}^2}/\operatorname{SE_{stat}^2}\to 0$.
Furthermore, the gradient of the bias in the likelihood evaluation, $C_{\gamma} n \big/ m$, leads to a bias of order $1/m$ in finding the location of $\theta$ that gives the maximum of the metamodel.
Taking $\mathcal{O}(\sqrt{n})\ll m(n)\ll \mathcal{O}(n)$ ensures the asymptotic bias in the estimator is negligible compared $\operatorname{SE_{stat}}$.
\end{proof}

\section*{Acknowledgements}
The authors would like to thank the anonymous reviewers, the Associate Editor, and the Editor-in-Chief for their constructive comments that greatly improved the quality of this paper. This research project was supported by NSF grant DMS-$1761603$.

\bibliography{ms}
\bigskip
\bigskip

\begin{table}[h!]
	\begin{center}
		\caption{Table of Notation}	\label{tab:TableOfNotation}
		\begin{tabular}{r c p{7cm} }
			\toprule
			$s(\theta)=s(\cdot\;;\theta):=\sqrt{p(\cdot\;;\theta)}$ && Square root of density $p(\cdot\;;\theta)$, Eqn. \eqref{eqn:density_def}.\\
			$\mathbf{l}(\theta)=\sum_{i=1}^n l(Y_i\;;\theta)$ && Log-likelihood of samples $(Y_1,\cdots,Y_n)$, Eqn. \eqref{eqn:Sample_loglikelihood}.\\
			$\dot{\mathbf{I}}(\theta)=2\frac{\dot{s}(\theta)}{s(\theta)}\mathbbm{1}_{\{s(\theta)>0\}}$ & & Variable defined in Definition \ref{def:2nd_smooth_point}.\\
			$\ddot{\mathbf{I}}(\theta)=2\frac{\ddot{s}(\theta)}{s(\theta)}\mathbbm{1}_{\{s(\theta)>0\}}$ & & Variable defined in Definition \ref{def:2nd_smooth_point}.\\
			$\dddot{\mathbf{I}}(\theta)=2\frac{\dddot{s}(\theta)}{s(\theta)}\mathbbm{1}_{\{s(\theta)>0\}}$ & & Variable defined in Assumption \ref{model_assumption}.\\
			$\ddddot{\mathbf{I}}(\theta)=2\frac{\ddddot{s}(\theta)}{s(\theta)}\mathbbm{1}_{\{s(\theta)>0\}}$ & & Variable defined in Assumption \ref{model_assumption}.\\
			${\mathcal{I}}(\theta)=E_{\theta}[ \dot{\mathbf{I}}(\theta)]^2$ &  &  The second moment of $\dot{\mathbf{I}}(\theta)$, Definition \ref{def:2nd_smooth_point}.\\
			$T_n=\left\{\frac{s(\theta+\delta_n t)}{s(\theta)}-1-\frac{\delta_n  \dot{s}(\theta)t}{s(\theta)} \right\}\mathbbm{1}_{\{s(\theta)>0\}}$ & & Key variable in this paper, Eqn. \eqref{eqn:Tn}.\\
			$A_n=\big\{\max_{1\leq i \leq n}|T_{ni}+\frac{1}{2}\delta_n t_n \dot{\mathbf{I}}(\theta)|<\eta\big\}$ && Truncated variable, Eqn. \eqref{eqn:An_classical}.\\
			$\widetilde{A}=\widetilde{A}_n(\epsilon)=\big\{\max_{1\leq i \leq n}|T_{ni}|<\epsilon\big\}$ & & Truncated variable, Eqn. \eqref{eqn:An}.\\
			$S_n(\theta)=\frac{1}{\sqrt{n}}\sum_{i=1}^n \dot{\mathbf{I}}(Y_i\;;\theta)$ & & Variable defined in Eqn. \eqref{eqn:S_n_theta}.\\
			$V_n(\theta)=\frac{1}{\sqrt{n}}\sum_{i=1}^n \left[ \ddot{\mathbf{I}}(Y_i \;;\theta) -E_{\theta}[\ddot{\mathbf{I}}(\theta)] \right]$ & & Variable defined in Eqn. \eqref{eqn:V_n_theta}.\\
			$U_n(\theta)=\frac{1}{\sqrt{n}}\sum_{i=1}^n \left[\dot{\mathbf{I}}^2(Y_i\;;\theta) -{\mathcal{I}}(\theta)\right]$ & & Variable defined in Eqn. \eqref{eqn:U_n_theta}.\\			
			$\widehat\theta^{\oneStepB}_n=\theta_n^{\ast}+\sqrt{n}\times \frac{S_n(\theta_n^{\ast})}{{\mathcal{I}}(\theta_n^{\ast})}$ && Estimator defined in Eqn. \eqref{eq:classical_onestep_est}.\\
			$\widehat{\theta}^{\oneStepA}_n=\theta_n^{\ast}+\delta_n \times \frac{\sqrt{n}\delta_nS_n(\theta_n^{\ast})}{n\delta_n^2 \mathcal{I}(\theta_n^{\ast})}$ && Estimator defined in Eqn. \eqref{eqn:rescale_onestep}.\\
			$\widehat\theta^{\text{MCLE}}_n$&& Estimator defined in Eqn. \eqref{eqn:MCLEestimator_def}.\\
			\bottomrule
		\end{tabular}
	\end{center}	
\end{table}

\newpage
\appendix

\section{Proof of Lemma \ref{lem:Tn_second_moment}}\label{app:Tn_second_moment}
\textbf{Proof of equation \eqref{eqn::Tn_second_moment1}}:
	Note that
	\begin{equation}
	\begin{split}
	& \hspace*{-0.4cm} E_{\theta}\left |T_n- \frac{1}{4}(\delta_n t_n)^2 \ddot{\mathbf{I}}(\theta)\right|^2\\
	=& \int \left(\left\{\frac{s(\theta+\delta_n t_n)}{s(\theta)}-1-\frac{1}{2}\delta_n t_n \dot{\mathbf{I}}(\theta) \right\}-\frac{1}{4}(\delta_n t_n)^2 \ddot{\mathbf{I}}(\theta)\right)^2 s^2(\theta)d\mu\\
	=& \left\lVert s(\theta+\delta_n t_n)-s(\theta)- \delta_n t_n \dot{s}(\theta)-\frac{1}{2}(\delta_n t_n)^2\ddot{s}(\theta)\right\rVert^2.
	\end{split}
	\end{equation}
	Recall that, by Taylor series, for any function $h$ with $(n+1)$-th derivative $h^{(n+1)}$ we have
	\begin{equation}
	\label{eqn:taylor_series}
	\begin{split}
	h(x)=&h(x_0)+(x-x_0)h'(x_0)+\frac{h''(x_0)}{2!}(x-x_0)^2+\cdots\\
	&+\frac{h^{(n)}(x_0)}{n!}(x-x_0)^n
	+\frac{1}{n!}\int_{x_0}^x (x-t)^nh^{(n+1)}(t)dt.
	\end{split}
	\end{equation}
	Hence, for $K$ compact, 
	\begin{equation}
	\begin{split}
	& \hspace*{-1.2cm} \sup_{\theta\in K}\frac{1}{(\delta_n t_n)^4} \left\lVert s(\theta+\delta_n t_n)-s(\theta)- \delta_n t_n \dot{s}(\theta)-\frac{1}{2}(\delta_n t_n)^2\ddot{s}(\theta)\right\rVert^2\\
	=& \sup_{\theta\in K}\frac{1}{(\delta_n t_n)^4} \left\lVert (\delta_n t_n)^2 \left(\int_0^1 (1-\lambda)\ddot{s}(\theta+\lambda \delta_n t_n)d\lambda -\frac{1}{2}\ddot{s}(\theta) \right) \right\rVert^2\\
	\leq &  \sup_{\theta\in K}\frac{1}{(\delta_n t_n)^4} \left\lVert (\delta_n t_n)^2 \int_0^1 (1-\lambda)\left(\ddot{s}(\theta+\lambda \delta_n t_n) -\ddot{s}(\theta) \right) d\lambda \right\rVert^2\\
	\leq &   \int_0^1 (1-\lambda)\sup_{\theta\in K}\left\lVert  \left(\ddot{s}(\theta+\lambda \delta_n t_n) -\ddot{s}(\theta) \right) \right\rVert^2 d\lambda\\
	\rightarrow & \;  0,
	\end{split}
	\end{equation}
	as $\delta_n\rightarrow 0$ and $0<|t_n|\leq M$, by the continuity of the map $\theta\rightarrow \ddot{s}(\theta)$.\\
	
\noindent	\textbf{Proof of equation \eqref{eqn::Tn_second_moment2}}:
	Note that
	\begin{align*}
	& \hspace*{-0.4cm} E_{\theta}\left |T_n- \frac{1}{4}(\delta_n t_n)^2 \ddot{\mathbf{I}}(\theta)-\frac{1}{12}(\delta_nt_n)^3\dddot{\mathbf{I}}(\theta)\right|^2\\
	=& \int \left(\left\{\frac{s(\theta+\delta_n t_n)}{s(\theta)}-1-\frac{1}{2}\delta_n t_n \dot{\mathbf{I}}(\theta) \right\}-\frac{1}{4}(\delta_n t_n)^2 \ddot{\mathbf{I}}(\theta)-\frac{1}{12}(\delta_n t_n)^3\dddot{\mathbf{I}}(\theta)\right)^2 s^2(\theta)d\mu\\
	=& \left\lVert s(\theta+\delta_n t_n)-s(\theta)- \delta_n t_n \dot{s}(\theta)-\frac{1}{2}(\delta_n t_n)^2\ddot{s}(\theta)-\frac{1}{3!}(\delta_n t_n)^3\dddot{s}(\theta)\right\rVert^2.
	\end{align*}
	Hence, for $K$ compact, 
	\begin{equation}
	\begin{split}
	& \hspace*{-1.1cm} \sup_{\theta\in K}\frac{1}{(\delta_n t_n)^6} \left\lVert s(\theta+\delta_n t_n)-s(\theta)- \delta_n t_n \dot{s}(\theta)-\frac{1}{2}(\delta_n t_n)^2\ddot{s}(\theta)-\frac{1}{3!}(\delta_n t_n)^3\dddot{s}(\theta)\right\rVert^2\\
	=& \sup_{\theta\in K}\frac{1}{(\delta_n t_n)^6} \left\lVert (\delta_n t_n)^3 \left(\int_0^1 \frac{1}{2} (1-\lambda)^2\dddot{s}(\theta+\lambda \delta_n t_n)d\lambda -\frac{1}{3!}\dddot{s}(\theta) \right) \right\rVert^2\\
	\leq &  \sup_{\theta\in K}\frac{1}{(\delta_n t_n)^6} \left\lVert (\delta_n t_n)^3 \frac{1}{2}\int_0^1 (1-\lambda)^2\left(\dddot{s}(\theta+\lambda \delta_n t_n) -\dddot{s}(\theta) \right) d\lambda \right\rVert^2\\
	\leq &   \int_0^1 \frac{1}{2}(1-\lambda)^2\sup_{\theta\in K}\left\lVert  \left(\dddot{s}(\theta+\lambda \delta_n t_n) -\dddot{s}(\theta) \right) \right\rVert^2 d\lambda\\
	\rightarrow & \;  0,
	\end{split}
	\end{equation}
	as $\delta_n\rightarrow 0$ and $0<|t_n|\leq M$, by the continuity of the map $\theta\rightarrow \dddot{s}(\theta)$.\\
	
\noindent	\textbf{Proof of equation \eqref{eqn::Tn_second_moment3}}:
	Note that
	\begin{align*}
	& \hspace*{-0.2cm} E_{\theta}\left |T_n- \frac{1}{4}(\delta_n t_n)^2 \ddot{\mathbf{I}}(\theta)-\frac{1}{12}(\delta_n t_n)^3\dddot{\mathbf{I}}(\theta)-\frac{1}{48}(\delta_n t_n)^4\ddddot{\mathbf{I}}(\theta)\right|^2\\
	=& \int \left(\left\{\frac{s(\theta+\delta_n t_n)}{s(\theta)}-1-\frac{1}{2}\delta_n t_n \dot{\mathbf{I}}(\theta) \right\}-\frac{1}{4}(\delta_n t_n)^2 \ddot{\mathbf{I}}(\theta)-\frac{1}{12}(\delta_n t_n)^3\dddot{\mathbf{I}}(\theta)\right.\\
	&\left.-\frac{1}{48}(\delta_n t_n)^4\ddddot{\mathbf{I}}(\theta)\right)^2 s^2(\theta)d\mu\\
	=& \left\lVert s(\theta+\delta_n t_n)-s(\theta)- \delta_n t_n \dot{s}(\theta)-\frac{1}{2}(\delta_n t_n)^2\ddot{s}(\theta)-\frac{1}{3!}(\delta_n t_n)^3\dddot{s}(\theta)-\frac{1}{4!}(\delta_n t_n)^4\ddddot{s}(\theta)\right\rVert^2.
	\end{align*}
	Hence, for $K$ compact, 
	\begin{align*}
	& \hspace*{-1.2cm} \sup_{\theta\in K}\frac{1}{(\delta_n t_n)^8} \left\lVert s(\theta+\delta_n t_n)-s(\theta)- \delta_n t_n \dot{s}(\theta)-\frac{1}{2}(\delta_n t_n)^2\ddot{s}(\theta)-\frac{1}{3!}(\delta_n t_n)^3\dddot{s}(\theta)\right.\\
	&\hspace*{1.2cm}\left.-\frac{1}{4!}(\delta_n t_n)^4\ddddot{s}(\theta)\right\rVert^2\\
	=& \sup_{\theta\in K}\frac{1}{(\delta_n t_n)^8} \left\lVert (\delta_n t_n)^4 \left(\int_0^1 \frac{1}{3!} (1-\lambda)^3\ddddot{s}(\theta+\lambda \delta_n t_n)d\lambda -\frac{1}{4!}\ddddot{s}(\theta) \right) \right\rVert^2\\
	\leq &  \sup_{\theta\in K}\frac{1}{(\delta_n t_n)^8} \left\lVert (\delta_n t_n)^4 \frac{1}{3!}\int_0^1 (1-\lambda)^3\left(\ddddot{s}(\theta+\lambda \delta_n t_n) -\ddddot{s}(\theta) \right) d\lambda \right\rVert^2\\
	\leq &   \int_0^1 \frac{1}{3!}(1-\lambda)^3\sup_{\theta\in K}\left\lVert  \left(\ddddot{s}(\theta+\lambda \delta_n t_n) -\ddddot{s}(\theta) \right) \right\rVert^2 d\lambda\\
	\rightarrow & \;  0,
	\end{align*}
	as $\delta_n\rightarrow 0$ and $0<|t_n|\leq M$, by the continuity of the map $\theta\rightarrow \ddddot{s}(\theta)$.

\section{Proof of Lemma \ref{lem:ddot_I2_ui}}\label{app:ddot_I2_ui}
\textbf{Proof of equation \eqref{eqn:ddot_I2_ui_1}: }	
We apply reductio ad absurdum here by assuming that there exist $\theta_n \in K$ and $\lambda_n \rightarrow \infty$ such that
\begin{equation}
\begin{split}\label{eqn:reductio_ad_absurdum1}
E_{\theta_n}\left[\left |\ddot{\mathbf{I}}(\theta_n)\right|^2 \mathbbm{1}_{\left\{\left |\ddot{\mathbf{I}}(\theta_n)\right|\geq \lambda_n\right\}}\right]>\alpha>0.
\end{split}
\end{equation}
Since $K$ is compact, we may assume that $\theta_n \rightarrow \theta \in K$. 

Recall that 
$$E_{\theta}[\ddot{\mathbf{I}}(\theta)]^2=E_{\theta}|\ddot{\mathbf{I}}(\theta)|^2=4\int \ddot{s}^2(\theta)\mathbbm{1}_{\left\{s(\theta)>0\right\}} d\mu.$$
We have
\begin{align*}
&\hspace*{-1.2cm} |\ddot{\mathcal{I}}(\theta+h)-E_{\theta}[\ddot{\mathbf{I}}(\theta)]^2|\\
=& 4\left| \int \ddot{s}^2(\theta+h)d\mu- \int \ddot{s}^2(\theta)d\mu \right|\\
\leq & 4\int |\ddot{s}(\theta+h)|\cdot |\ddot{s}(\theta+h)-\ddot{s}(\theta)|d\mu+4\int |\ddot{s}(\theta)|\cdot |\ddot{s}(\theta+h)-\ddot{s}(\theta)|d\mu\\
\leq & 4\left\lVert \ddot{s}(\theta+h) \right\rVert \cdot  \left\lVert \ddot{s}(\theta+h)-\ddot{s}(\theta) \right\rVert + 4\left\lVert \ddot{s}(\theta) \right\rVert \cdot \left\lVert \ddot{s}(\theta+h)-\ddot{s}(\theta) \right\rVert\\
\rightarrow & \;  0,
\end{align*}
as $h \rightarrow 0$, by the continuity of $\theta \rightarrow \ddot{s}(\theta)$ in $\mathcal{L}_2(\mu)$. Hence, we have that 
$E_{\theta}[\ddot{\mathbf{I}}(\theta)]^2$ is continuous in $\Theta \in \mathbb{R}$.
It follows that 
$$E_{\theta_n}\left |\ddot{\mathbf{I}}(\theta_n)\right|^2 \rightarrow E_{\theta}|\ddot{\mathbf{I}}(\theta)|^2.$$
Hence, by Lemma A.$7.2.$B of \cite{bickel1993efficient} and Assumption \ref{model_assumption}, the sequence of random variable $\{ |\ddot{\mathbf{I}}(\theta_n)|^2\}$ is uniformly integrable, hence
$$\lim_{n\rightarrow \infty}E_{\theta_n}\left |\ddot{\mathbf{I}}(\theta_n)\right|^2 \mathbbm{1}_{\left\{\left |\ddot{\mathbf{I}}(\theta_n)\right|\geq \lambda_n\right\}}=0,$$
which contradicts to \eqref{eqn:reductio_ad_absurdum1}. Therefore, one has
$$\lim_{\lambda \rightarrow \infty} \sup_{\theta\in K} E_{\theta}\left[|\ddot{\mathbf{I}}(\theta)|^2 \mathbbm{1}_{\left\{|\ddot{\mathbf{I}}(\theta)|\geq \lambda\right\}}\right]=0,$$
as desired. \\

\noindent\textbf{Proof of equation \eqref{eqn:ddot_I2_ui_2}: }
We apply reductio ad absurdum here by assuming that there exist $\theta_n \in K$ and $\lambda_n \rightarrow \infty$ such that
\begin{equation}
\begin{split}\label{eqn:reductio_ad_absurdum2}
E_{\theta_n}\left[\left |\dot{\mathbf{I}}(\theta_n)\right|^4 \mathbbm{1}_{\left\{\left |\dot{\mathbf{I}}(\theta_n)\right|\geq \lambda_n\right\}}\right]>\alpha>0.
\end{split}
\end{equation}
Since $K$ is compact, we may assume that $\theta_n \rightarrow \theta \in K$. 

We firstly show that $E_{\theta}\left(\dot{\mathbf{I}}(\theta)\right)^4$ is continuous in $\Theta \in \mathbb{R}$.
By the continuity of $\theta \rightarrow p(\theta)$ which is given by Proposition $1$.(i) on page $13$ of \cite{bickel1993efficient},  
$\theta \rightarrow s(\theta)$ is continuous which is given by the definition of $s(\cdot)$. Let $\phi(\theta)=s^{-1}(\theta)$ for $s(\theta)\neq 0$, and then $\phi(\theta)$ is continuous. Note that
$\dot{\phi}(\theta)=-s^{-2}(\theta)\dot{s}(\theta)\mathbbm{1}_{\{s(\theta)>0\}} $. By the continuity of 
$\theta \rightarrow \dot{s}(\theta)$ which is given by Definition $2$.(ii) on page $12$ of \cite{bickel1993efficient} and our Assumption \ref{model_assumption}, we know that $\dot{\phi}(\theta)$ is continuous with respect to $\theta$.
Note that
\begin{align*}
\frac{1}{16}E_{\theta}\left(\dot{\mathbf{I}}(\theta)\right)^4
=\int \left(\frac{ \dot{s}^4(\theta)}{s^2(\theta)}\right) \mathbbm{1}_{\{s(\theta)>0\}}  d\mu=-\int \dot{s}^3(\theta)\dot{\phi}(\theta) d\mu.
\end{align*}
 We have that as $h \rightarrow 0$,
\begin{align*}
&\hspace*{-1.1cm}  \left|\int \dot{s}^3(\theta+h)\dot{\phi}(\theta+h) d\mu-\int \dot{s}^3(\theta)\dot{\phi}(\theta) d\mu\right|\\
\leq & \int |\dot{s}^3(\theta+h)|\cdot |\dot{\phi}(\theta+h)-\dot{\phi}(\theta)|d\mu+\int |\dot{\phi}(\theta)|\cdot |\dot{s}^3(\theta+h)-\dot{s}^3(\theta)|d\mu\nonumber\\
\leq & \left\lVert \dot{s}^3(\theta+h) \right\rVert \cdot  \left\lVert \dot{\phi}(\theta+h)-\dot{\phi}(\theta) \right\rVert + \left\lVert \dot{\phi}(\theta) \right\rVert \cdot \left\lVert \dot{s}^3(\theta+h)-\dot{s}^3(\theta) \right\rVert\nonumber\\
\rightarrow & \;  0,
\end{align*}
by the continuity of $\theta \rightarrow \dot{s}(\theta)$ in $\mathcal{L}_2(\mu)$ and the continuity of $\theta \rightarrow \dot{\phi}(\theta)$ in $\mathcal{L}_2(\mu)$. 
It follows that 
$$E_{\theta_n}\left |\dot{\mathbf{I}}(\theta_n)\right|^4 \rightarrow E_{\theta}|\dot{\mathbf{I}}(\theta)|^4.$$
Hence, by Lemma A.$7.2.$B of \cite{bickel1993efficient} and Assumption \ref{model_assumption}, the sequence of random variable $\{ |\dot{\mathbf{I}}(\theta_n)|^4\}$ is uniformly integrable, hence
$$\lim_{n\rightarrow \infty}E_{\theta_n}\left |\dot{\mathbf{I}}(\theta_n)\right|^4 \mathbbm{1}_{\left\{\left |\dot{\mathbf{I}}(\theta_n)\right|\geq \lambda_n\right\}}=0,$$
which contradicts to \eqref{eqn:reductio_ad_absurdum2}. Therefore, one has
$$\lim_{\lambda \rightarrow \infty} \sup_{\theta\in K} E_{\theta}\left[|\dot{\mathbf{I}}(\theta)|^4 \mathbbm{1}_{\left\{|\dot{\mathbf{I}}(\theta)|\geq \lambda\right\}}\right]=0,$$
as desired. \\

\noindent\textbf{Proof of equation \eqref{eqn:ddot_I2_ui_3}: }	
We apply reductio ad absurdum here by assuming that there exist $\theta_n \in K$ and $\lambda_n \rightarrow \infty$ such that
\begin{equation}
\begin{split}\label{eqn:reductio_ad_absurdum3}
E_{\theta_n}\left[\left |\ddddot{\mathbf{I}}(\theta_n)\right| \mathbbm{1}_{\left\{\left |\ddddot{\mathbf{I}}(\theta_n)\right|\geq \lambda_n\right\}}\right]>\alpha>0.
\end{split}
\end{equation}
Since $K$ is compact, we may assume that $\theta_n \rightarrow \theta \in K$. 
Note that
\begin{align*}
E_{\theta}[\ddddot{\mathbf{I}}(\theta)]
=2\int\ddddot{s}(\theta)s(\theta) d\mu
\end{align*}
and
\begin{equation*}
\begin{split}
&\hspace*{-1.1cm} \left|\int|\ddddot{s}(\theta+h)|s(\theta+h) d\mu-\int |\ddddot{s}(\theta)|s(\theta) d\mu\right|\\
\leq & \int |\ddddot{s}(\theta+h)|\cdot |{s}(\theta+h)-{s}(\theta)|d\mu+\int |{s}(\theta)|\cdot \big||\ddddot{s}(\theta+h)|-|\ddddot{s}(\theta)|\big|d\mu\\
\leq & \left\lVert \ddddot{s}(\theta+h) \right\rVert \cdot  \left\lVert {s}(\theta+h)-{s}(\theta) \right\rVert + \left\lVert {s}(\theta) \right\rVert \cdot \big\lVert |\ddddot{s}(\theta+h)|-|\ddddot{s}(\theta)| \big\rVert\\
\rightarrow & \;  0,
\end{split}
\end{equation*}
as $h \rightarrow 0$, by the continuity of $\theta \rightarrow \ddddot{s}(\theta)$ in $\mathcal{L}_2(\mu)$, the continuity of $\theta \rightarrow s(\theta)$ in $\mathcal{L}_2(\mu)$, and the fact that the absolute value function of a continuous function is continuous. 
It follows that 
$$E_{\theta_n}\left |\ddddot{\mathbf{I}}(\theta_n)\right| \rightarrow E_{\theta}\left |\ddddot{\mathbf{I}}(\theta)\right|.$$
Hence, by Lemma A.$7.2.$B of \cite{bickel1993efficient} and Assumption \ref{model_assumption}, the sequence of random variable $\{ |\ddddot{\mathbf{I}}(\theta_n)|\}$ is uniformly integrable, hence
$$\lim_{n\rightarrow \infty}E_{\theta_n}\left |\ddddot{\mathbf{I}}(\theta_n)\right| \mathbbm{1}_{\left\{\left |\ddddot{\mathbf{I}}(\theta_n)\right|\geq \lambda_n\right\}}=0,$$
which contradicts to \eqref{eqn:reductio_ad_absurdum3}. Therefore, one has
$$\lim_{\lambda \rightarrow \infty} \sup_{\theta\in K} E_{\theta}\left[\left |\ddddot{\mathbf{I}}(\theta)\right| \mathbbm{1}_{\left\{\left |\ddddot{\mathbf{I}}(\theta)\right|\geq \lambda\right\}}\right]=0,$$
as desired. \\

\noindent\textbf{Proof of equation \eqref{eqn:ddot_I2_ui_4}: }	
We apply reductio ad absurdum here by assuming that there exist $\theta_n \in K$ and $\lambda_n \rightarrow \infty$ such that
\begin{equation}
\begin{split}\label{eqn:reductio_ad_absurdum4}
E_{\theta_n}\left[\left |\dddot{\mathbf{I}}(\theta_n)\dot{\mathbf{I}}(\theta_n)\right| \mathbbm{1}_{\left\{\left |\dddot{\mathbf{I}}(\theta_n)\dot{\mathbf{I}}(\theta_n)\right|\geq \lambda_n\right\}}\right]>\alpha>0.
\end{split}
\end{equation}
Since $K$ is compact, we may assume that $\theta_n \rightarrow \theta \in K$. 
Note that
\begin{align*}
E_{\theta}[\dddot{\mathbf{I}}(\theta)\dot{\mathbf{I}}(\theta)]
=4\int\dddot{s}(\theta)\dot{s}(\theta) d\mu
\end{align*}
and
\begin{equation*}
\begin{split}
&\hspace*{-1.1cm} \left|\int|\dddot{s}(\theta+h)\dot{s}(\theta+h)| d\mu-\int|\dddot{s}(\theta)\dot{s}(\theta)| d\mu\right|\\
\leq & \int |\dddot{s}(\theta+h)|\cdot \big||\dot{s}(\theta+h)|-|\dot{s}(\theta)|\big|d\mu+\int |\dot{s}(\theta)|\cdot \big||\dddot{s}(\theta+h)|-|\dddot{s}(\theta)|\big|d\mu\\
\leq & \left\lVert \dddot{s}(\theta+h) \right\rVert \cdot  \big\lVert |\dot{s}(\theta+h)|-|\dot{s}(\theta)| \big\rVert + \left\lVert \dot{s}(\theta) \right\rVert \cdot \big\lVert |\dddot{s}(\theta+h)|-|\dddot{s}(\theta)| \big\rVert\\
\rightarrow & \;  0,
\end{split}
\end{equation*}
as $h \rightarrow 0$, by the continuity of $\theta \rightarrow \dddot{s}(\theta)$ in $\mathcal{L}_2(\mu)$, the continuity of $\theta \rightarrow \dot{s}(\theta)$ in $\mathcal{L}_2(\mu)$, and the fact that the absolute value function of a continuous function is continuous. 
It follows that 
$$E_{\theta_n}\left |\dddot{\mathbf{I}}(\theta_n)\dot{\mathbf{I}}(\theta_n)\right| \rightarrow E_{\theta}|\dddot{\mathbf{I}}(\theta)\dot{\mathbf{I}}(\theta)|.$$
Hence, by Lemma A.$7.2.$B of \cite{bickel1993efficient} and Assumption \ref{model_assumption}, 
$$\lim_{n\rightarrow \infty}E_{\theta_n}\left |\dddot{\mathbf{I}}(\theta_n)\dot{\mathbf{I}}(\theta_n)\right| \mathbbm{1}_{\left\{\left |\dddot{\mathbf{I}}(\theta_n)\dot{\mathbf{I}}(\theta_n)\right|\geq \lambda_n\right\}}=0,$$
which contradicts to \eqref{eqn:reductio_ad_absurdum4}. Therefore, one has
$$\lim_{\lambda \rightarrow \infty} \sup_{\theta\in K} E_{\theta}\left[|\dddot{\mathbf{I}}(\theta)\dot{\mathbf{I}}(\theta)| \mathbbm{1}_{\left\{|\dddot{\mathbf{I}}(\theta)\dot{\mathbf{I}}(\theta)|\geq \lambda\right\}}\right]=0,$$
as desired. 

\section{Proof of Proposition \ref{prop:sum_T_secondmoment}}\label{app:sum_T_secondmoment}
We firstly prove the case that $r=0$. Note that
\begin{equation*}
\begin{split}
&\hspace*{-0.9cm} \sum_{i=1}^n \left|T_{ni}^2-\frac{1}{16}\delta_n^4t_n^4 E_{\theta}[\ddot{\mathbf{I}}(\theta)]^2\right|\\
\leq & \sum_{i=1}^n \left|T_{ni}^2 -\frac{1}{16}\delta_n^4t_n^4(\ddot{\mathbf{I}}(Y_i\;;\theta))^2\right|
+\sum_{i=1}^n \left|\frac{1}{16}\delta_n^4t_n^4(\ddot{\mathbf{I}}(Y_i\;;\theta))^2 
-\frac{1}{16}\delta_n^4 t_n^4 E_{\theta}[\ddot{\mathbf{I}}(\theta)]^2\right|.
\end{split}
\end{equation*}
On one hand, we have that
\begin{equation*}
\begin{split}
E_{\theta}\sum_{i=1}^n \left|T_{ni}^2 -\frac{1}{16} \delta_n^4t_n^4(\ddot{\mathbf{I}}(Y_i\;;\theta))^2\right|= & nE_{\theta}\left|T_{n}^2 - \left(\frac{1}{4}\delta_n^2 t_n^2\ddot{\mathbf{I}}(\theta)\right)^2 \right|
\end{split}
\end{equation*}
By Lemma \ref{lem:Tn_second_moment}, uniformly in $|t_n| \leq M$ and in $\theta\in K \subset \Theta$ for $K$ compact, one has
$$E_{\theta}\left |T_n-\frac{1}{4} (\delta_n t_n)^2 \ddot{\mathbf{I}}(\theta)\right|^2=o(\delta_n^4),$$
which implies
$$E_{\theta}(\sqrt{n}T_{n})^2=o(1)$$
for $t_n=0$,
and 
$$E_{\theta}\left(\frac{\sqrt{n}}{t_n^2}T_{n} -\frac{1}{4}\sqrt{n}\delta_n^2\ddot{\mathbf{I}}(\theta)\right)^2=o(1)$$
for $0<|t_n| \leq M$.
It follows, by Lemma A.$7.1$ in \cite{bickel1993efficient}, that
\begin{equation}
\label{eqn:nTn_moment}
E_{\theta}\left|(\sqrt{n}T_{n})^2 -\frac{1}{16}(t_n^2\sqrt{n}\delta_n^2\ddot{\mathbf{I}}(\theta))^2 \right|=o(1).
\end{equation}
On the other hand, we have that
\begin{equation*}
\left(\frac{1}{n}\sum_{i=1}^n \ddot{\mathbf{I}}(Y_i\;;\theta)^2
- E_{\theta}[\ddot{\mathbf{I}}(\theta)]^2\right) \xrightarrow[]{a.s.} 0,
\end{equation*}
uniformly in $|t_n| \leq M$ and in $\theta\in K \subset \Theta$ for $K$ compact,
by the definition of $E_{\theta}[\ddot{\mathbf{I}}(\theta)]^2$, Lemma \ref{lem:ddot_I2_ui}, and 
Chung's uniform strong law of large number which can be seen in Theorem A.$7.3$ in \cite{bickel1993efficient}. Hence, 
\begin{equation*}
n\delta_n^4 t_n^4\left(\frac{1}{n}\sum_{i=1}^n \ddot{\mathbf{I}}(Y_i\;;\theta)^2
-  E_{\theta}[\ddot{\mathbf{I}}(\theta)]^2\right) \xrightarrow[]{a.s.} 0,
\end{equation*}
and then
\begin{equation}
\label{sum_T_secondmoment_inequality}
\sum_{i=1}^n \left|T_{ni}^2-\frac{1}{16}\delta_n^4t_n^4 E_{\theta}[\ddot{\mathbf{I}}(\theta)]^2\right|\xrightarrow[]{p} 0,
\end{equation}
uniformly in $|t_n| \leq M$ and in $\theta\in K \subset \Theta$ for $K$ compact, in $P_{\theta}$ probability.
Next, for $r>0$, considering that
\begin{equation*}
\begin{split}
\max_{1\leq i \leq n}\left|T_{ni}^2 - \frac{1}{16}\delta_n^4t_n^4E_{\theta}[\ddot{\mathbf{I}}(\theta)]^2\right|
\leq \sum_{i=1}^n \left|T_{ni}^2 - \frac{1}{16}\delta_n^4t_n^4E_{\theta}[\ddot{\mathbf{I}}(\theta)]^2\right|
\end{split}
\end{equation*}
and 
\begin{equation*}
\begin{split}
&\sum_{i=1}^n \left|T_{ni}^2 - \frac{1}{16}\delta_n^4t_n^4E_{\theta}[\ddot{\mathbf{I}}(\theta)]^2\right|^{1+r}\\
&\leq \max_{1\leq i \leq n}\left|T_{ni}^2 - \frac{1}{16}\delta_n^4t_n^4E_{\theta}[\ddot{\mathbf{I}}(\theta)]^2\right|^r \sum_{i=1}^n \left|T_{ni}^2 - \frac{1}{16}\delta_n^4t_n^4E_{\theta}[\ddot{\mathbf{I}}(\theta)]^2\right|,
\end{split}
\end{equation*}
by \eqref{sum_T_secondmoment_inequality},
we have that in $P_{\theta}$ probability,
$$\sum_{i=1}^n \left|T_{ni}^2 - \frac{1}{16}\delta_n^4 t_n^4E_{\theta}[\ddot{\mathbf{I}}(\theta)]^2\right|^{1+r}\xrightarrow[]{p} 0\quad \text{for}\; r>0,$$  
uniformly in $\theta \in K$ and $|t_n| \leq M$.

\section{Proof of Proposition \ref{prop:Tmulti_dotS}}\label{app:Tmulti_dotS}
We firstly consider the case that $l=2$ and $1\leq k< 6$. Note that
\begin{equation*}
\begin{split}
\sum_{i=1}^n \left| T_{ni}^2\left(\delta_n t_n \dot{\mathbf{I}}(Y_i\;;\theta)\right)^k \right|
\leq & \sum_{i=1}^n \left| \left(T_{ni}^2 - \frac{1}{16}\delta_n^4 t_n^4E_{\theta}[\ddot{\mathbf{I}}(\theta)]^2\right) \left(\delta_n t_n \dot{\mathbf{I}}(Y_i\;;\theta)\right)^k\right|\\
&+\sum_{i=1}^n \left| \frac{1}{16}\delta_n^4 t_n^4E_{\theta}[\ddot{\mathbf{I}}(\theta)]^2 \left(\delta_n t_n \dot{\mathbf{I}}(Y_i\;;\theta)\right)^k\right|,
\end{split}
\end{equation*}
and by H\"older's inequality in the counting measure we have
\begin{equation*}
\begin{split}
&\hspace*{-1.5cm}  \sum_{i=1}^n \left|\left(T_{ni}^2 - \frac{1}{16}\delta_n^4 t_n^4E_{\theta}[\ddot{\mathbf{I}}(\theta)]^2\right) \left(\delta_n t_n \dot{\mathbf{I}}(Y_i\;;\theta)\right)^k\right|\\
\leq & \left[\sum_{i=1}^n \left|T_{ni}^2 - \frac{1}{16}\delta_n^4 t_n^4E_{\theta}[\ddot{\mathbf{I}}(\theta)]^2\right|^{\frac{6}{6-k}}\right]^{\frac{6-k}{6}} \left[\sum_{i=1}^n \left|\delta_n t_n \dot{\mathbf{I}}(Y_i\;;\theta)\right|^{k \times \frac{6}{k}} \right]^{\frac{k}{6}}.
\end{split}
\end{equation*}
We can see that, to complete the proof,  
it suffices to show that uniformly in $|t_n| \leq M$ and in $\theta\in K \subset \Theta$ for $K$ compact, in $P_{\theta}$ probability,
$$\sum_{i=1}^n \left|T_{ni}^2 - \frac{1}{16}\delta_n^4 t_n^4E_{\theta}[\ddot{\mathbf{I}}(\theta)]^2\right|^{1+r}\xrightarrow[]{p} 0\quad \text{for any}\; r\geq 0,$$  
$$\sum_{i=1}^n \left|\delta_n t_n \dot{\mathbf{I}}(Y_i\;;\theta)\right|^m\xrightarrow[]{p} 0,\quad \text{for any}\; m\in \{5,6\},$$ and 
$$\sum_{i=1}^n \left|\frac{1}{16}\delta_n^4 t_n^4E_{\theta}[\ddot{\mathbf{I}}(\theta)]^2 \left(\delta_n t_n \dot{\mathbf{I}}(Y_i\;;\theta)\right)^k \right| \xrightarrow[]{p} 0\quad\text{for any}\; 1\leq k< 6.$$
The above conditions are satisfied by Propositions \ref{prop:sum_T_secondmoment} and  \ref{prop:dotS_fifth}, and noting that for $|t_n| \leq M$ and $\epsilon'>0$
\begin{equation*}
\begin{split}
&\hspace*{-1.5cm}P_{\theta}\left( \sum_{i=1}^n \left|\delta_n^4 t_n^4E_{\theta}[\ddot{\mathbf{I}}(\theta)]^2 \left(\delta_n t_n \dot{\mathbf{I}}(Y_i\;;\theta)\right)^k \right|>\epsilon' \right)\\
\leq & \frac{1}{\epsilon'} E_{\theta}\left( \sum_{i=1}^n \left|\delta_n^4 t_n^4E_{\theta}[\ddot{\mathbf{I}}(\theta)]^2 \left(\delta_n t_n \dot{\mathbf{I}}(Y_i\;;\theta)\right)^k \right|\right)\\
\leq & \frac{(\delta_n)^k M^{k+4}}{\epsilon'}\left|E_{\theta}[\ddot{\mathbf{I}}(\theta)]^2\right| E_{\theta}\left|\dot{\mathbf{I}}(\theta)\right|^k\\
\rightarrow & \;  0.
\end{split}
\end{equation*}

The result can be extended to any $l>2$ and $k\geq 6$, by noting that
\begin{align*}
&\sum_{i=1}^n \left|T_{ni}^l \left(\delta_n t_n \dot{\mathbf{I}}(Y_i\;;\theta)\right)^k \right|\\
&\leq 
\bigg(\max_{1\leq i \leq n} |T_{ni}|^{l-2}\bigg)\bigg(\max_{1\leq i \leq n}\left|\left(\delta_n t_n \dot{\mathbf{I}}(Y_i\;;\theta)\right)^5\right|\bigg) \sum_{i=1}^n \left|T_{ni}^2 \left(\delta_n t_n \dot{\mathbf{I}}(Y_i\;;\theta)\right)^{k-5} \right|\\
&\leq 
\bigg(\max_{1\leq i \leq n} |T_{ni}|^{l-2}\bigg)\bigg(\sum_{i=1}^n\left|\left(\delta_n t_n \dot{\mathbf{I}}(Y_i\;;\theta)\right)^5\right|\bigg) \sum_{i=1}^n \left|T_{ni}^2 \left(\delta_n t_n \dot{\mathbf{I}}(Y_i\;;\theta)\right)^{k-5} \right|
\end{align*}
and using Propositions \ref{prop:A_complement} and \ref{prop:dotS_fifth}.

\section{Proof of Proposition \ref{prop:U_CLT}}\label{app:U_CLT}
Let $f: \mathbb{R}\rightarrow \mathbb{R}$ be bounded and continuous. To prove 
\eqref{eqn:prop_CLT},
it suffices to show that
\begin{equation}
\label{eqn:prop_CLT_suffice}
\sup_{\theta\in K}\left(E_{\theta}f\left(\frac{1}{\sqrt{n}}\sum_{i=1}^n U(Y_i\;;\theta)\right)-E_{\theta}f(Z)\right)\rightarrow 0,
\end{equation}
as $n \rightarrow \infty$, where $Z\sim N(0,E_{\theta}[U^2(\theta)])$. Now suppose that \eqref{eqn:prop_CLT_suffice} fails and then there exists $\theta_n\in K$ such that
\begin{equation}
\label{eqn:prop_CLT_fail}
\limsup_{n\rightarrow \infty}\left(E_{\theta_n}f\left(\frac{1}{\sqrt{n}}\sum_{i=1}^n U(Y_i\;;\theta_n)\right)-E_{\theta_n}f(Z)\right)> 0.
\end{equation}
Since $K$ is compact, without loss of generality, we assume that $\theta_n\rightarrow \theta \in K$.

By Lemma \ref{lem:ddot_I2_ui}, we have
	$$\lim_{\lambda \rightarrow \infty} \sup_{\theta\in K} E_{\theta}\left[|\dot{\mathbf{I}}(\theta)|^4 \mathbbm{1}_{\left\{|\dot{\mathbf{I}}(\theta)|\geq \lambda\right\}}\right]=0,$$
which by Jensen's inequality implies 
	$$\lim_{\lambda \rightarrow \infty} \sup_{\theta\in K} E_{\theta}\left[|\dot{\mathbf{I}}(\theta)|^2 \mathbbm{1}_{\left\{|\dot{\mathbf{I}}(\theta)|\geq \lambda\right\}}\right]=0.$$
Hence, by the definition of $U(Y_i\;;\theta)$, we have
\begin{equation}
\label{eqn:U_u_ui}
\lim_{\lambda\rightarrow \infty}\sup_{\theta\in K}E_{\theta}\left| U(Y_i\;;\theta) \right|^2\mathbbm{1}_{\{\left| U(Y_i\;;\theta) \right|>\lambda\}}=0.
\end{equation}
Applying the bound of Theorem A.$7.4$ on page $470$ of \cite{bickel1993efficient} to the mean $0$ variance $1$ random variable $tU(Y_i\;;\theta_n)/\sigma_n$ where $t\in\mathbb{R}$ and $$\sigma_n^2=t^2E_{\theta}[U^2(Y_i\;;\theta_n)],$$ by \eqref{eqn:U_u_ui}, we have
$$\mathbf{L}_{\theta_n}\left(\frac{t}{\sqrt{n}}\sum_{i=1}^n U(\theta_n)\bigg/\sigma_n\right)\rightarrow N(0,1),
$$
for all $\theta_n$. Now we see that if the continuity of $E_{\theta_n}[U^2(Y_i\;;\theta_n)]$ with respect to $\theta_n$ holds, we can have
$$\mathbf{L}_{\theta_n}\left(\frac{t}{\sqrt{n}}\sum_{i=1}^n U(\theta_n)\right)\rightarrow N(0,t^2 E_{\theta}[U^2(\theta)]),
$$
and then in view of the Cram{\'e}r-Wold device we can obtain the contradiction to \eqref{eqn:prop_CLT_fail} which completes the proof.
Hence, for the reason that the continuity of $E_{\theta}[\dot{\mathbf{I}}^2(\theta)]$ is given in Proposition A.$5.3$.A of \cite{bickel1993efficient}, the continuity of $E_{\theta}[\dot{\mathbf{I}}^4(\theta)]$ with respect to $\theta$ (covered in the proof of Lemma \ref{lem:ddot_I2_ui}) completes the proof.

\section{Proof of Proposition \ref{prop:V_CLT}}\label{app:V_CLT}
Let $f: \mathbb{R}\rightarrow \mathbb{R}$ be bounded and continuous. To prove 
\eqref{eqn:prop_CLT2},
it suffices to show that
\begin{equation}
\label{eqn:prop_CLT_suffice2}
\sup_{\theta\in K}\left(E_{\theta}f\left(\frac{1}{\sqrt{n}}\sum_{i=1}^n V(Y_i\;;\theta)\right)-E_{\theta}f(Z)\right)\rightarrow 0,
\end{equation}
as $n \rightarrow \infty$, where $Z\sim N(0,\Var[\ddot{\mathbf{I}}(\theta)])$. Now suppose that \eqref{eqn:prop_CLT_suffice2} fails and then there exists $\theta_n\in K$ such that
\begin{equation}
\label{eqn:prop_CLT_fail2}
\limsup_{n\rightarrow \infty}\left(E_{\theta_n}f\left(\frac{1}{\sqrt{n}}\sum_{i=1}^n V(Y_i\;;\theta_n)\right)-E_{\theta_n}f(Z)\right)> 0.
\end{equation}
Since $K$ is compact, without loss of generality, we assume that $\theta_n\rightarrow \theta \in K$. 

By the proof of Lemma \ref{lem:ddot_I2_ui}, we know that 
the sequence of random variables $\{ |\ddot{\mathbf{I}}(\theta_n)|^2\}_n$ is uniformly integrable, and then it is not hard to obtain that the sequence of random variables $\{ |V(\theta_n)|^2\}_n$ is uniformly integrable which further yields
\begin{equation}
\label{eqn:V_u_ui}
\lim_{\lambda \rightarrow \infty} \sup_{\theta\in K} E_{\theta}\left[\left |V(\theta)\right|^2 \mathbbm{1}_{\left\{\left |V(\theta)\right|\geq \lambda\right\}}\right]=0.
\end{equation}
Applying the bound of Theorem A.$7.4$ on page $470$ of \cite{bickel1993efficient} to the mean $0$ variance $1$ random variable $tV(Y_i\;;\theta_n)/\sigma_n$ where $t\in\mathbb{R}$ and $$\sigma_n^2=t^2E_{\theta_n}[V^2(Y_i\;;\theta_n)],$$ by \eqref{eqn:V_u_ui}, we have
$$\mathbf{L}_{\theta_n}\left(\frac{t}{\sqrt{n}}\sum_{i=1}^n V(Y_i\;;\theta_n)\bigg/\sigma_n\right)\rightarrow N(0,1),
$$
for all $\theta_n$. Now we see that if the continuity of $E_{\theta_n}[V^2(Y_i\;;\theta_n)]$ with respect to $\theta_n$ holds, we can have
$$\mathbf{L}_{\theta_n}\left(\frac{t}{\sqrt{n}}\sum_{i=1}^n V(Y_i\;;\theta_n)\right)\rightarrow N(0,t^2 E_{\theta}[V^2(\theta)]),
$$
and then in view of the Cram{\'e}r-Wold device we can obtain the contradiction to \eqref{eqn:prop_CLT_fail} which completes the proof.

Recall that we have shown that $E_{\theta}[\ddot{\mathbf{I}}(\theta)]^2$ is continuous in $\Theta \in \mathbb{R}$, in the proof of Lemma \ref{lem:ddot_I2_ui}.
Considering that $$\int \ddot{s}(\theta) s(\theta)\mathbbm{1}_{\left\{s(\theta)=0\right\}} d\mu=0,$$
we have
$$\frac{1}{2}E_{\theta}[\ddot{\mathbf{I}}(\theta)]=\int \ddot{s}(\theta) s(\theta)\mathbbm{1}_{\left\{s(\theta)>0\right\}} d\mu=\int \ddot{s}(\theta) s(\theta) d\mu.$$
Then, as $h \rightarrow 0$, by the continuity of $\theta \rightarrow \ddot{s}(\theta)$ and $\theta \rightarrow s(\theta)$ in $\mathcal{L}_2(\mu)$,
\begin{align*}
&\hspace*{-1.4cm}  \left| \int \ddot{s}(\theta+h)s(\theta+h)d\mu- \int \ddot{s}(\theta)s(\theta)d\mu \right|\\
\leq & \int |\ddot{s}(\theta+h)|\cdot |s(\theta+h)-s(\theta)|d\mu+\int |s(\theta)|\cdot |\ddot{s}(\theta+h)-\ddot{s}(\theta)|d\mu\nonumber\\
\leq & C\left\lVert \ddot{s}(\theta+h) \right\rVert \cdot  \left\lVert s(\theta+h)-s(\theta) \right\rVert + C\left\lVert s(\theta) \right\rVert \cdot \left\lVert \ddot{s}(\theta+h)-\ddot{s}(\theta) \right\rVert\nonumber\\
\rightarrow & \;  0.
\end{align*}
 Hence, we have the continuity of $E_{\theta_n}[V^2(Y_i\;;\theta_n)]$ with respect to $\theta_n$, as desired.

\section{Proof of Lemma \ref{lemma:order_determinate}}\label{app:order_determinate}

We can rewrite $\operatorname{det}\left(\overline{\mathbf{X}}^T\overline{\mathbf{X}}\right)$ as follows
\begin{align*}
\operatorname{det}\left(\overline{\mathbf{X}}^T\overline{\mathbf{X}}\right)=K\mathcal{A}-
\left(\sum_{j=-J}^J (jn^{-1/4})\right)\mathcal{B}+\left(\sum_{j=-J}^J (jn^{-1/4})^2\right)\mathcal{C}-\left(\sum_{j=-J}^J (jn^{-1/4})^3\right)\mathcal{D},
\end{align*}
where 
\begin{align*}
\mathcal{A}=&\operatorname{det}\left( \begin{array}{ccc} 
 \sum_{j=-J}^J (jn^{-1/4})^2 &\sum_{j=-J}^J (jn^{-1/4})^3 &\sum_{j=-J}^J (jn^{-1/4})^4\\
 \sum_{j=-J}^J (jn^{-1/4})^3 &\sum_{j=-J}^J (jn^{-1/4})^4 &\sum_{j=-J}^J (jn^{-1/4})^5\\
 \sum_{j=-J}^J (jn^{-1/4})^4 &\sum_{j=-J}^J (jn^{-1/4})^5 &\sum_{j=-J}^J (jn^{-1/4})^6 \end{array} \right)\\
 =& \sum_{j=-J}^J (jn^{-1/4})^2\left(\sum_{j=-J}^J (jn^{-1/4})^4 \sum_{j=-J}^J (jn^{-1/4})^6-\left(\sum_{j=-J}^J (jn^{-1/4})^5\right)^2\right)\\
 &- \sum_{j=-J}^J (jn^{-1/4})^3\left(\sum_{j=-J}^J (jn^{-1/4})^3 \sum_{j=-J}^J (jn^{-1/4})^6-\sum_{j=-J}^J (jn^{-1/4})^4\sum_{j=-J}^J (jn^{-1/4})^5\right)\\
&+ \sum_{j=-J}^J (jn^{-1/4})^4\left(\sum_{j=-J}^J (jn^{-1/4})^3 \sum_{j=-J}^J (jn^{-1/4})^5-\left(\sum_{j=-J}^J (jn^{-1/4})^4\right)^2\right),
\end{align*}

\begin{align*}
\mathcal{B}=&\operatorname{det}\left( \begin{array}{ccc} 
\sum_{j=-J}^J (jn^{-1/4})  &\sum_{j=-J}^J (jn^{-1/4})^3 &\sum_{j=-J}^J (jn^{-1/4})^4\\
\sum_{j=-J}^J (jn^{-1/4})^2  &\sum_{j=-J}^J (jn^{-1/4})^4 &\sum_{j=-J}^J (jn^{-1/4})^5\\
\sum_{j=-J}^J (jn^{-1/4})^3  &\sum_{j=-J}^J (jn^{-1/4})^5 &\sum_{j=-J}^J (jn^{-1/4})^6 \end{array} \right)\\
 =& \sum_{j=-J}^J (jn^{-1/4})\left(\sum_{j=-J}^J (jn^{-1/4})^4 \sum_{j=-J}^J (jn^{-1/4})^6-\left(\sum_{j=-J}^J (jn^{-1/4})^5\right)^2\right)\\
 &- \sum_{j=-J}^J (jn^{-1/4})^3\left(\sum_{j=-J}^J (jn^{-1/4})^2 \sum_{j=-J}^J (jn^{-1/4})^6-\sum_{j=-J}^J (jn^{-1/4})^3\sum_{j=-J}^J (jn^{-1/4})^5\right)\\
&+ \sum_{j=-J}^J (jn^{-1/4})^4\left(\sum_{j=-J}^J (jn^{-1/4})^2 \sum_{j=-J}^J (jn^{-1/4})^5-\sum_{j=-J}^J (jn^{-1/4})^3\sum_{j=-J}^J (jn^{-1/4})^4\right),
\end{align*}

\begin{align*}
\mathcal{C}=&\operatorname{det}\left( \begin{array}{ccc}
\sum_{j=-J}^J (jn^{-1/4}) & \sum_{j=-J}^J (jn^{-1/4})^2  &\sum_{j=-J}^J (jn^{-1/4})^4\\
\sum_{j=-J}^J (jn^{-1/4})^2 & \sum_{j=-J}^J (jn^{-1/4})^3 &\sum_{j=-J}^J (jn^{-1/4})^5\\
\sum_{j=-J}^J (jn^{-1/4})^3 & \sum_{j=-J}^J (jn^{-1/4})^4  &\sum_{j=-J}^J (jn^{-1/4})^6 \end{array} \right)\\
 =& \sum_{j=-J}^J (jn^{-1/4})\left(\sum_{j=-J}^J (jn^{-1/4})^3 \sum_{j=-J}^J (jn^{-1/4})^6-\sum_{j=-J}^J (jn^{-1/4})^4\sum_{j=-J}^J (jn^{-1/4})^5\right)\\
 &- \sum_{j=-J}^J (jn^{-1/4})^2\left(\sum_{j=-J}^J (jn^{-1/4})^2 \sum_{j=-J}^J (jn^{-1/4})^6-\sum_{j=-J}^J (jn^{-1/4})^3\sum_{j=-J}^J (jn^{-1/4})^5\right)\\
&+ \sum_{j=-J}^J (jn^{-1/4})^4\left(\sum_{j=-J}^J (jn^{-1/4})^2 \sum_{j=-J}^J (jn^{-1/4})^4-\left(\sum_{j=-J}^J (jn^{-1/4})^3\right)^2\right),
\end{align*}
and
\begin{align*}
\mathcal{D}=&\operatorname{det}\left( \begin{array}{ccc} 
\sum_{j=-J}^J (jn^{-1/4}) & \sum_{j=-J}^J (jn^{-1/4})^2 &\sum_{j=-J}^J (jn^{-1/4})^3 \\
\sum_{j=-J}^J (jn^{-1/4})^2 & \sum_{j=-J}^J (jn^{-1/4})^3 &\sum_{j=-J}^J (jn^{-1/4})^4 \\
\sum_{j=-J}^J (jn^{-1/4})^3 & \sum_{j=-J}^J (jn^{-1/4})^4 &\sum_{j=-J}^J (jn^{-1/4})^5  \end{array} \right)
\\
 =& \sum_{j=-J}^J (jn^{-1/4})\left(\sum_{j=-J}^J (jn^{-1/4})^3 \sum_{j=-J}^J (jn^{-1/4})^5-\left(\sum_{j=-J}^J (jn^{-1/4})^4\right)^2\right)\\
 &- \sum_{j=-J}^J (jn^{-1/4})^2\left(\sum_{j=-J}^J (jn^{-1/4})^2 \sum_{j=-J}^J (jn^{-1/4})^5-\sum_{j=-J}^J (jn^{-1/4})^3\sum_{j=-J}^J (jn^{-1/4})^4\right)\\
&+ \sum_{j=-J}^J (jn^{-1/4})^3\left(\sum_{j=-J}^J (jn^{-1/4})^2 \sum_{j=-J}^J (jn^{-1/4})^4-\left(\sum_{j=-J}^J (jn^{-1/4})^3\right)^2\right).
\end{align*}
 By the fact that
\begin{eqnarray*}
&\sum_{i=1}^n i=\frac{n(n+1)}{2}, &\sum_{i=1}^n i^2=\frac{n(n+1)(2n+1)}{6},\\
&\sum_{i=1}^n i^3=\frac{n^2(n+1)^2}{4}, &\sum_{i=1}^n i^4=\frac{n(n+1)(2n+1)(3n^2+3n-1)}{30},\\
&\sum_{i=1}^n i^5=\frac{n^2(n+1)^2(2n^2+2n-1)}{12}, &\sum_{i=1}^n i^6=\frac{n(n+1)(2n+1)(3n^4+6n^3-3n+1)}{42},
\end{eqnarray*}
summarizing the above results, we have $\operatorname{det}\left(\overline{\mathbf{X}}^T\overline{\mathbf{X}}\right)=\mathcal{O}(n^{-3})$.

\section{Proof of Lemma \ref{lemma:order_adjugate}}\label{app:order_adjugate}
We have that
\begin{align*}
&\hspace*{-0.5cm}\operatorname{adj}\left(\overline{\mathbf{X}}^T\overline{\mathbf{X}}\right)_{22}\\
=&\operatorname{det}\left( \begin{array}{cccc} K  &\sum_{j=-J}^J (jn^{-1/4})^2 &\sum_{j=-J}^J (jn^{-1/4})^3 \\
\sum_{j=-J}^J (jn^{-1/4})^2  &\sum_{j=-J}^J (jn^{-1/4})^4 &\sum_{j=-J}^J (jn^{-1/4})^5\\
\sum_{j=-J}^J (jn^{-1/4})^3 &\sum_{j=-J}^J (jn^{-1/4})^5 &\sum_{j=-J}^J (jn^{-1/4})^6 \end{array} \right)\\
=&K\left(\sum_{j=-J}^J (jn^{-1/4})^4 \sum_{j=-J}^J (jn^{-1/4})^6-\left(\sum_{j=-J}^J (jn^{-1/4})^5\right)^2\right)\\
&-\sum_{j=-J}^J (jn^{-1/4})^2 \left(\sum_{j=-J}^J (jn^{-1/4})^2 \sum_{j=-J}^J (jn^{-1/4})^6-\sum_{j=-J}^J (jn^{-1/4})^3\sum_{j=-J}^J (jn^{-1/4})^5\right)\\
&+\sum_{j=-J}^J (jn^{-1/4})^3 \left(\sum_{j=-J}^J (jn^{-1/4})^2 \sum_{j=-J}^J (jn^{-1/4})^5-\sum_{j=-J}^J (jn^{-1/4})^3\sum_{j=-J}^J (jn^{-1/4})^4\right)\\
=&\; \mathcal{O}(n^{-5/2})
\end{align*}
and 
\begin{align*}
&\hspace*{-0.5cm}\operatorname{adj}\left(\overline{\mathbf{X}}^T\overline{\mathbf{X}}\right)_{22}\\
=&\operatorname{det}\left( \begin{array}{cccc} K  &\sum_{j=-J}^J (jn^{-1/4})^2 &\sum_{j=-J}^J (jn^{-1/4})^3 \\
\sum_{j=-J}^J (jn^{-1/4})^2  &\sum_{j=-J}^J (jn^{-1/4})^4 &\sum_{j=-J}^J (jn^{-1/4})^5\\
\sum_{j=-J}^J (jn^{-1/4})^3 &\sum_{j=-J}^J (jn^{-1/4})^5 &\sum_{j=-J}^J (jn^{-1/4})^6 \end{array} \right)\\
=&K\left(\sum_{j=-J}^J (jn^{-1/4})^4 \sum_{j=-J}^J (jn^{-1/4})^6-\left(\sum_{j=-J}^J (jn^{-1/4})^5\right)^2\right)\\
&-\sum_{j=-J}^J (jn^{-1/4})^2 \left(\sum_{j=-J}^J (jn^{-1/4})^2 \sum_{j=-J}^J (jn^{-1/4})^6-\sum_{j=-J}^J (jn^{-1/4})^3\sum_{j=-J}^J (jn^{-1/4})^5\right)\\
&+\sum_{j=-J}^J (jn^{-1/4})^3 \left(\sum_{j=-J}^J (jn^{-1/4})^2 \sum_{j=-J}^J (jn^{-1/4})^5-\sum_{j=-J}^J (jn^{-1/4})^3\sum_{j=-J}^J (jn^{-1/4})^4\right)\\
=&\; \mathcal{O}(n^{-2}).
\end{align*}

\end{document}